\documentclass[reqno,  12pt]{amsart}
\usepackage{extarrows}
\usepackage{ragged2e}
\usepackage{amssymb}
\usepackage{amsfonts}
\usepackage{amsmath}
\usepackage{mathrsfs}
\usepackage{diagbox}
\usepackage{color,  soul}
\usepackage{bbm}
\usepackage{amsthm, graphicx,  empheq}
\usepackage{framed}
\usepackage[backref,  colorlinks,  linkcolor=red,  anchorcolor=green,  citecolor=blue]{hyperref}

\usepackage{cancel}
\usepackage{ulem}
\usepackage{arydshln}

\usepackage{tikz}
\usetikzlibrary{arrows,backgrounds,snakes,shapes}

\usepackage{booktabs}
\usepackage{makecell}
\usepackage[linesnumbered,ruled]{algorithm2e}

\usepackage{lineno}  

\newtheorem{lemma}{Lemma}[section]
\newtheorem{theorem}{Theorem}[section]
\newtheorem{definition}{Definition}[section]
\newtheorem{remark}{Remark}[section]
\numberwithin{equation}{section}

\makeatletter
\@namedef{subjclassname@2020}{\textup{2020} Mathematics Subject Classification}
\makeatother

\usetikzlibrary{calc,trees,positioning,arrows,chains,shapes.geometric,%
    decorations.pathreplacing,decorations.pathmorphing,shapes,%
    matrix,shapes.symbols}
\tikzset{
>=stealth',
punktchain/.style={
	rectangle, 
	rounded corners, 
	draw=black, very thick,
	text width = 8em, 
	minimum height=2em, 
	text centered
                        },
yuan/.style={
	ellipse, 
	rounded corners, 
	draw=black, very thick,
	text width = 8em, 
	minimum height=2em, 
	text centered
                        },
line/.style={draw, thick, <-},
element/.style={
	tape,
	top color=white,
	bottom color=blue!50!black!60!,
	minimum width=8em,
	draw=blue!40!black!90, very thick,
	text width=10em, 
	minimum height=3.5em, 
	text centered, 
	on chain},
every join/.style={->, thick,shorten >=1pt},
decoration={brace},
tuborg/.style={decorate},
tubnode/.style={midway, right=2pt},
punkt/.style={
	rectangle, 
	rounded corners, 
	draw=black, very thick,
	text width=12em, 
	minimum height=3em, 
	text centered},
}

\begin{document}


\title[Integral Invariant-type Weak Solution]
{Well-posedness and Regularity of the Integral Invariant Model from Linear Scalar Transport Equation}
\footnote{This is a revised version of the manuscript originally submitted to \textit{Journal of Mathematical Analysis and Applications} on Jan 14, 2026. 
Following peer review, the paper has been modified and is currently under re-review. 
The previous preprint version (arXiv: 2503.07028) has been updated to reflect these changes. }
\author{Zhengrong Xie}

\address[Z. Xie]{School of Mathematical Sciences,  
East China Normal University, Shanghai 200241, China}\email{\tt xzr\_nature@163.com;\ 52265500018@stu.ecnu.edu.cn}



\date{\today}

\begin{abstract}
  An integral invariant model derived from the coupling of the transport equation and its adjoint equation is investigated. 
  Despite extensive research on the numerical implementation of this model, 
  no studies have yet explored the well-posedness and regularity of the model itself. 
  To address this gap, firstly, a comprehensive mathematical definition is formulated as a Cauchy initial value problem for the integral invariant model. 
  This formulation preserves essential background information derived from relevant numerical algorithms. 
  In the above definition, we directly evolve the time-dependent test function $\psi(\mathbf{x},t)$ through explicit construction rather than solving the adjoint equation, 
  which enables reducing the required regularity of the test function $\Psi(\mathbf{x})$ from $C^1(\Omega)$ to $L^2(\Omega)$, contributing to stability proof.  
  The challenge arising from the mismatch of integration domains on both sides of the model's equivalent form is overcome through the compact support property of test functions. 
  For any arbitrary time instant $t^{*}\in[0,T]$, an abstract function $\mathcal{U}(\lambda)$ taking values in the Banach space $L^2(\mathbb{R}^d)$ is initially constructed on the entire space $\mathbb{R}^d$ via the Riesz representation theorem. 
  Subsequently, this function is properly restricted to the time-dependent bounded domain $\widetilde{\Omega}(t)$ through multiplication by the characteristic function.
  The existence of this model's solution in $L^{1}([0,T],L^2(\widetilde{\Omega}(t)))$ is then rigorously established. 
  Furthermore, by judiciously selecting test functions $\Psi$, the stability of the integral invariant model is proved, from which the uniqueness naturally follows. 
  Finally, when the initial value \(\widetilde{U}_0 \in L^{2}(\widetilde{\Omega}(0))\),  
  the temporal integrability of the model over $[0,T]$ can be enhanced to \(L^{\infty}([0,T],L^2(\widetilde{\Omega}(t)))\). 
\end{abstract}

\allowbreak
\allowdisplaybreaks

\maketitle

\tableofcontents 

\section{Introduction}\label{sec-Introduction}
Consider the integral invariant formulation
\begin{equation}\label{eq-IntegralInvariantFormulation}
  \frac{\mathrm{d}}{\mathrm{d}t} \int_{\widetilde{\Omega}(t)} U \psi \, d\mathbf{x} = 0,
\end{equation}
which is derived from the linear scalar transport equation
\begin{equation}\label{eq-scalar-linear-transport}
  U_t + \nabla \cdot (\mathbf{A} U) = 0 \quad\text{in} \quad \Omega\times[0,T] 
\end{equation}
and its adjoint equation(dual equation) \cite{ref-HEC93, ref-CRH90, ref-RC02}
\begin{equation}\label{eq-adjoint}
  \psi_t + \mathbf{A} \cdot \nabla \psi = 0 \quad\text{in} \quad \Omega\times[0,T] 
\end{equation}
via the Leibniz rule (differentiation formula for integrals with variable limits, see Appendix \ref{appendix-Differentiation of a Definite Integral with Variable Limits}) 
or the Reynolds transport theorem in fluid mechanics \cite{ref-GNQ14}. 
Here, $U$ represents a physical quantity such as mass, energy, or particle number density, 
while $\psi$ denotes a time-dependent test function. This formulation differs fundamentally from conventional finite element or discontinuous Galerkin frameworks where test functions are typically time-independent. 
\( T \) is a positive constant, and \( \Omega \subset \mathbb{R}^d \) is a bounded domain. 
The vector function \( \mathbf{A}(\mathbf{x},t) = (a_1(\mathbf{x},t), a_2(\mathbf{x},t), \ldots, a_d(\mathbf{x},t))^{\text{T}} \) is assumed to be continuous with respect to both \( \mathbf{x} \) and \( t \), and its first-order partial derivatives are also continuous.  
The characteristic equation associated with Equation \eqref{eq-adjoint} is given by 
\begin{align}\label{eq-characteristic}
\frac{\mathrm{d}\mathbf{x}}{\mathrm{d}t} = \mathbf{A}(\mathbf{x},t),
\end{align}
and \(\widetilde{\Omega}(t)\) represents the characteristic space-time region emanating from the characteristic lines, 
as illustrated in Figure \ref{Fig.CharacteristicDynamicDomain}.
\begin{figure}[htbp]
  \begin{center}
    \begin{minipage}{0.49\linewidth}
      \centerline{\includegraphics[width=1\linewidth]{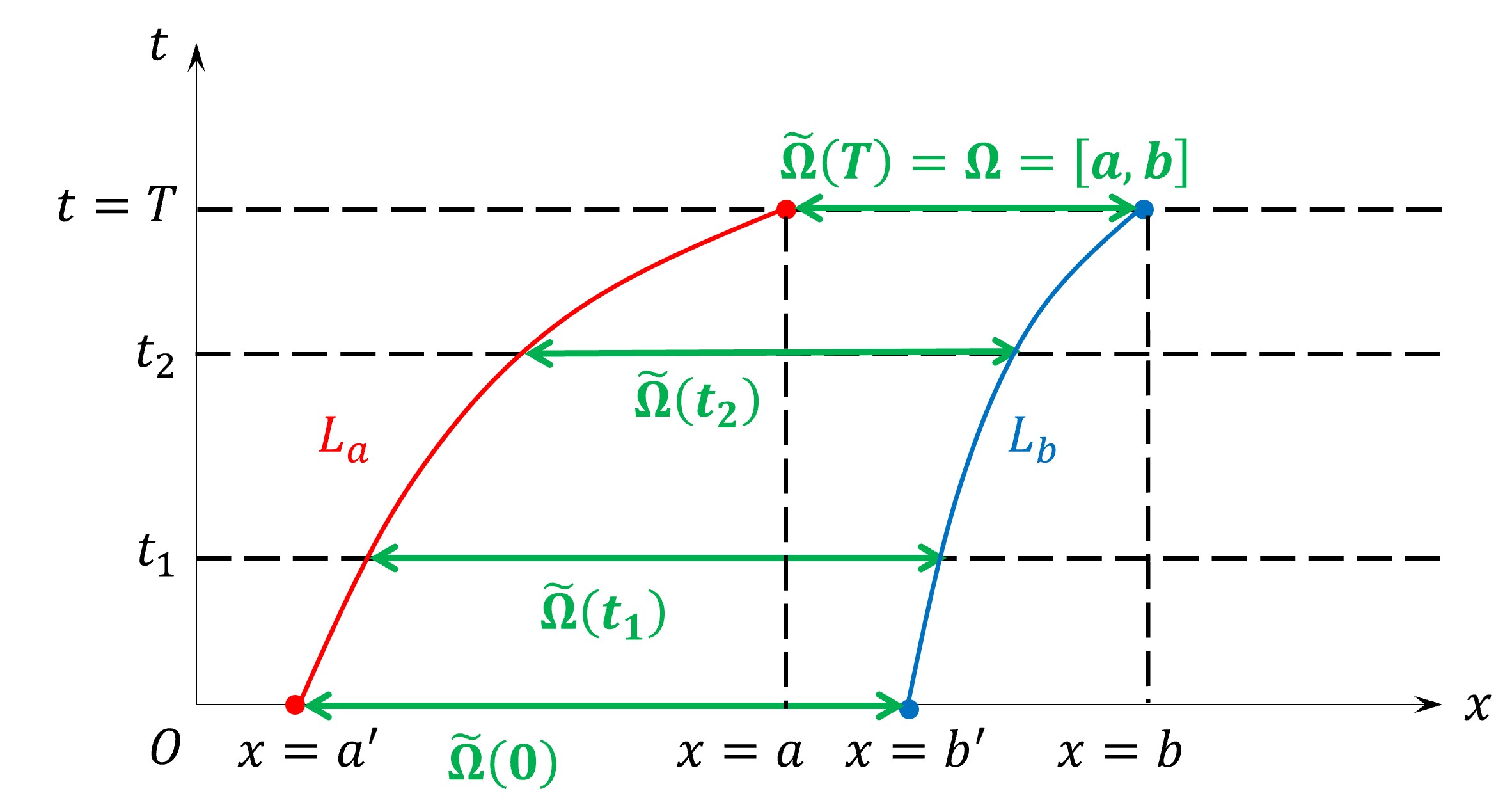}}
    \end{minipage}
    \hfill
    \begin{minipage}{0.49\linewidth}
      \centerline{\includegraphics[width=1\linewidth]{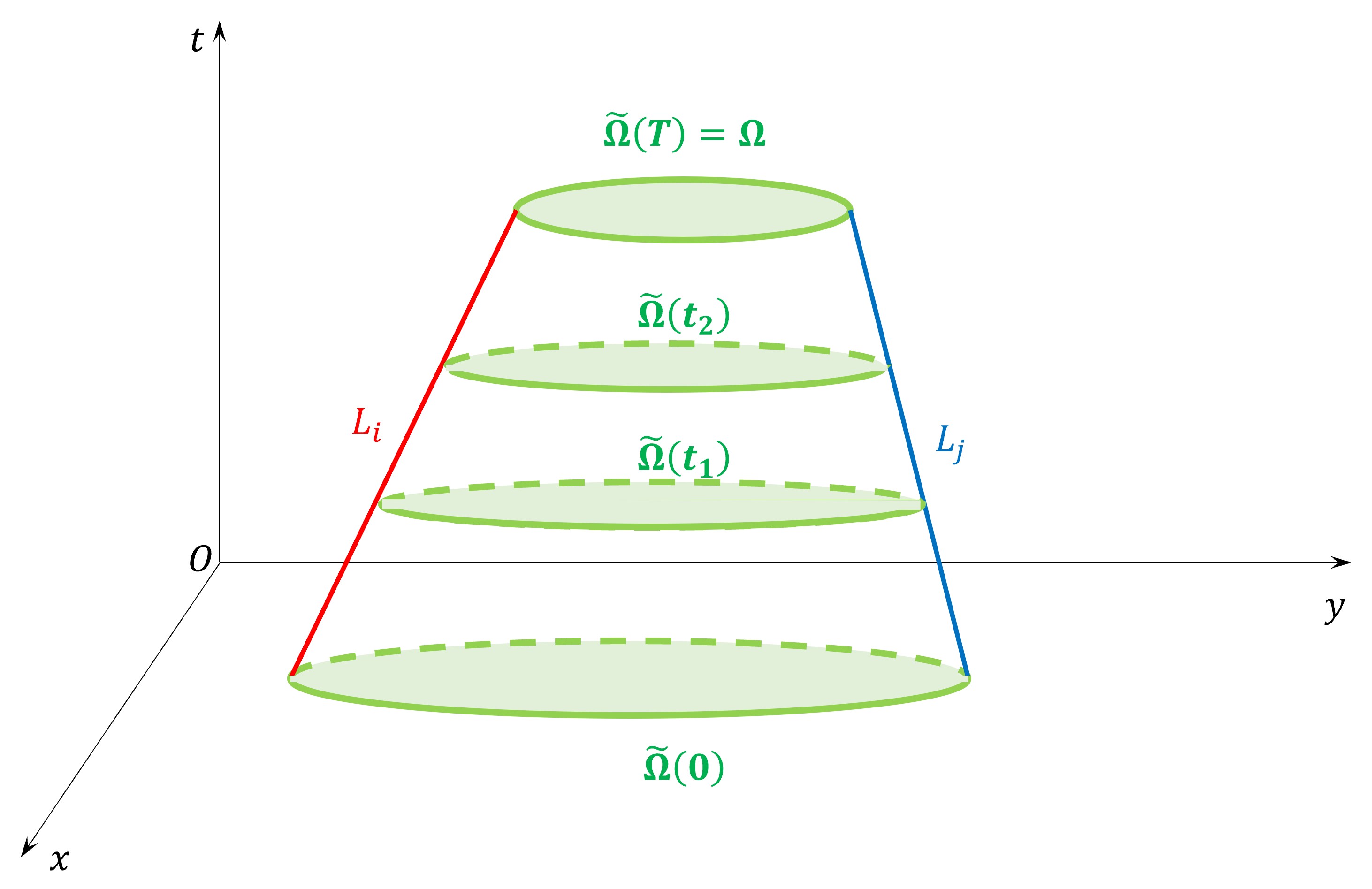}}
    \end{minipage}
    \vfill
    \begin{minipage}{0.49\linewidth}
      \centerline{\tiny{(a)\ $\Omega\subset\mathbb{R}$}}
    \end{minipage}
    \hfill
    \begin{minipage}{0.49\linewidth}
      \centerline{\tiny{(b)\  $\Omega\subset\mathbb{R}^2$}}
    \end{minipage}
  \end{center}
\caption{\tiny{Characteristic\ Dynamic\ Domain. }}
\label{Fig.CharacteristicDynamicDomain}
\end{figure}
\par
An essential condition for the validity of the integral invariant formulation \eqref{eq-IntegralInvariantFormulation} is evidently the non-intersection of characteristic lines. 
Therefore, if we denote the time at which discontinuities appear in the adjoint equation \eqref{eq-adjoint} 
(i.e., the intersection of characteristic lines) as \(T^{*}\), we require
\begin{align}
T < T^{*}\ (T^{*} \leq \infty),
\end{align}
where \(T^{*} = \infty\) implies that the characteristic lines never intersect, and the entire flow field remains smooth without discontinuities, e.g., incompressible flow with $\nabla\cdot\mathbf{A}=0$. 
In this case, no upper limit is imposed on \(T\).
\par
The linear scalar transport equation\eqref{eq-scalar-linear-transport} has significant applications across multiple disciplines including microscale drug delivery, plasma confinement studies, atmospheric pollutant dispersion, climate system modeling, and sediment transport dynamics. 
Meeting critical engineering demands for \textbf{long-time simulations}, \textbf{large-step computations}, and \textbf{unconditionally stable schemes} \textbf{is achieved by discretizing the integral invariant formulation} \eqref{eq-IntegralInvariantFormulation}, 
as demonstrated by the development of the \textbf{semi-Lagrangian discontinuous Galerkin (SLDG)} method\cite{ref-CGQ17,ref-CBQ20,ref-CGQ19}. 
{\color{black}{Although numerous numerical studies have implemented this model, the \textbf{fundamental mathematical analysis} of the integral invariant formulation \eqref{eq-IntegralInvariantFormulation}--particularly its well-posedness (existence, uniqueness, continuous dependence) and spatial-temporal regularity--remains an open problem. 
{\color{red}{(Rodolfo Bermejo's work on the Lagrange-Galerkin method targets real-world, complex fluid-dynamics problems represented by the Euler or Navier-Stokes equations \cite{ref-CCB21}; it does not directly address the model linear scalar transport equation or the foundational PDE weak-solution theory that we consider here. Moreover, our research setting is not fully Lagrangian but rather emphasizes a semi-Lagrangian framework.)}}
From a finite element perspective, this model\eqref{eq-IntegralInvariantFormulation} can be interpreted as a \textbf{non-conventional weak formulation} of the transport equation\eqref{eq-scalar-linear-transport} within a moving domain and time-dependent test functions. 
Just as the Lax-Milgram theorem underpins finite element methods by establishing well-posedness for Galerkin variational problems of elliptic PDEs, 
proving analogous results for this integral invariant framework would provide rigorous mathematical justification for numerical algorithms relying on its discretization, such as the aforementioned SLDG method. 
Moreover, the model occupies a unique position in PDE theory: it is neither a classical elliptic/parabolic/hyperbolic equation nor an integro-differential equation (e.g., Kirchhoff type\cite{ref-Kirchhoff83}). Investigating this structure could expand the horizons of modern PDE analysis by introducing a new category of \textbf{evolution equations defined through integral invariants}, potentially inspiring novel analytical techniques for transport-dominated systems. }}
\par
A rigorous analysis of well-posedness and regularity fundamentally depends on first establishing a \textbf{comprehensive mathematical formulation} for the integral invariant model. 
Beyond the principal equation \eqref{eq-IntegralInvariantFormulation}, this entails specifying initial conditions at a designated time slice, potential boundary conditions, and prescribing an appropriate solution space.  
Critical to this formulation, as developed in Subsection \ref{Sec-def-WeakSol}, is the avoidance of potential ill-posedness inherent in the adjoint equation's Cauchy final-value problem. 
This challenge involves developing quantitative descriptions for both the time-evolution of the domain $\widetilde{\Omega}(t)$ and the propagation dynamics of test functions $\psi(\mathbf{x},t)$. 
\par
The main work of this paper is as follows: Firstly,  after establishing the comprehensive definition of the integral invariant model, we utilize the compact support property of
test functions to overcome the difficulty caused by the time-varying integration domain and then employ the Riesz representation theorem and 
the theory of abstract functions valued in Banach spaces to prove the existence in $L^{1}([0,T],L^2(\widetilde{\Omega}(t)))$.  
Secondly, by skillfully selecting the test function \(\Psi\), we demonstrate the continuous dependence on initial data(stability) for solutions of the integral invariant model, 
which naturally leads to its uniqueness. 
Following the well-posedness analysis of the integral invariant model, we further investigate its regularity.  
By repeatedly utilizing an intermediate result from the existence proof——$\|U(\mathbf{x},t^*)\|^q_{L^2(\widetilde{\Omega}(t^*))} \lesssim \|\widetilde{U}_0\|^q_{L^2(\widetilde{\Omega}(0))},~2 \leq q < \infty$, 
we can conclude that when the initial value $\widetilde{U}_0(\mathbf{x})\in L^2(\widetilde{\Omega}(0))$, the integral invariant model can indeed achieve higher integrability in time $[0,T]$, specifically
\( U(\mathbf{x},t) \in L^q([0,T],L^2(\widetilde{\Omega}(t))) \) for \( 2 \leq q < \infty \).   
Finally, with the aid of the established Lemma \ref{lemma-Lp-Linfty} from previous work\cite{ref-CHenSX18-lemma}, we can further obtain:
\( U(\mathbf{x},t) \in L^{\infty}([0,T],L^2(\widetilde{\Omega}(t))) \) with
\( \|U(\mathbf{x},t)\|_{L^{\infty}([0,T],L^2(\widetilde{\Omega}(t)))}=\lim_{q \to \infty}\|U(\mathbf{x},t)\|_{L^{q}([0,T],L^2(\widetilde{\Omega}(t)))} \). 
\par 
The paper is organized as follows: Section \ref{sec-HYP-Preliminaries-Notations} introduces some preliminaries and the notations defined in this paper, which will be employed in the subsequent definitions and proofs. 
Section \ref{Sec-def-WeakSol} constructs comprehensive mathematical formulation for the integral invariant model and Section \ref{sec-WeakSol-Existence-Uniqueness-Stability} successively proves its existence, stability, and uniqueness. 
The regularity enhancement of this model is addressed in Section \ref{sec-WeakSol-regularity-t-Linfty}.  
All lemmas required for the proofs of the theorems in Section \ref{sec-WeakSol-Existence-Uniqueness-Stability} and Section \ref{sec-WeakSol-regularity-t-Linfty} are listed in Section \ref{sec-lemmas}. 
Finally, Section \ref{sec-Conclusion} summarizes the paper.

\section{Basic Assumptions, Preliminaries and Notations}\label{sec-HYP-Preliminaries-Notations}
First, assume that the velocity field $\mathbf{A}(\mathbf{x},t)$ in the linear scalar transport equation\eqref{eq-scalar-linear-transport} satisfies a uniformly bounded gradient condition, i.e., 
{\color{red}{there exists a constant $L_A > 0$ such that }}
\begin{align}\label{eq-gardA-UniformBoundedness}
  \sup_{\mathbf{x} \in \mathbb{R}^d, t \in [0, T]} \left\| \nabla_{\mathbf{x}} \mathbf{A}(\mathbf{x}, t) \right\|_{\mathrm{op}} \leq L_A,  
\end{align}
where $\nabla_{\mathbf{x}} \mathbf{A}(\mathbf{x}, t)$ is the Jacobian matrix of $\mathbf{A}$ with respect to the spatial variables $\mathbf{x}$.
$\|\cdot\|_{\mathrm{op}}$ denotes the operator norm of a matrix, defined as:
$$
\left\| \nabla_{\mathbf{x}} \mathbf{A} \right\|_{\mathrm{op}} = \sup_{\|\mathbf{v}\|=1} \left\| \left( \nabla_{\mathbf{x}} \mathbf{A} \right) \mathbf{v} \right\|. 
$$
This norm is equivalent to the largest singular value (spectral norm) of the Jacobian matrix. 
It is important to note that the \textbf{uniformly bounded gradient} condition implies that $\mathbf{A}$ satisfies the \textbf{linear growth condition}. 
Specifically, there exist continuous functions $\alpha(t)$ and $\beta(t)$ such that:
\begin{align}\label{eq-A-linearGrowth}
  \|\mathbf{A}(\mathbf{x},t)\| \leq \alpha(t)\cdot\|\mathbf{x}\| + \beta(t), \quad \forall (\mathbf{x},t)\in\mathbb{R}^d\times[0,T]. 
\end{align}
Furthermore, if the gradient of $\mathbf{A}$ is uniformly bounded, then \textbf{the divergence of $\mathbf{A}$ is necessarily uniformly bounded,} i.e., there exists $M_A>0$ such that 
\begin{align}\label{eq-div-UniformBoundedness}
   \left\| \nabla_{\mathbf{x}} \cdot \mathbf{A}(\mathbf{x}, t) \right\|_{L^{\infty}(\mathbb{R}^d)} \leq M_A, ~\forall t \in [0, T].   
\end{align}
Here, a feasible choice for the upper bound \(M_A\) is \(M_A = d L_A\). (Inspired by the relationship between the trace of a matrix and its row norm and operator norm) 
\par
{\color{red}{
Appendix \ref{appendix-A-HYP-Examples} presents standard SLDG test cases, which serve to validate the aforementioned underlying assumptions for the velocity field $\mathbf{A}(\mathbf{x},t)$. 
}}
\par
{\color{red}{As noted in Section \ref{sec-Introduction}, $\mathbf{A}(\mathbf{x},t)$ has \textbf{continuous first-order partial derivatives}, ensuring that both the Cauchy initial value problem and the terminal value problem for $\frac{\mathrm{d}\mathbf{s}}{\mathrm{d}\tau}=\mathbf{A}(\mathbf{s},\tau)$ are \textbf{locally well-posed} (on the intervals $[0,\delta]$ and $[T-\delta,T]$, respectively, with $\delta$ sufficiently small).  
Furthermore, when combined with the \textbf{linear growth constraint}~\eqref{eq-A-linearGrowth} resulting from the \textbf{uniformly bounded gradient condition}~\eqref{eq-gardA-UniformBoundedness}, the solutions to the Cauchy initial/terminal value problems for $\frac{\mathrm{d}\mathbf{s}}{\mathrm{d}\tau}=\mathbf{A}(\mathbf{s},\tau)$ can be extended to the entire interval $[0,T]$, i.e., they become \textbf{globally well-posed}. }} 
Therefore, for any fixed time $t \in [0,T]$ we can define the flow mapping $\operatorname{D}^{\mathbf{A}}_{T \to t}$: 
\begin{equation}\label{eq-FlowMap-tildeOmega}
  \begin{aligned}
    \operatorname{D}^{\mathbf{A}}_{T \to t}:~\Omega &\rightarrow \widetilde{\Omega}(t) \\
    \mathbf{x}_T &\mapsto \mathbf{x}_t=\mathbf{x}_T-\int_{t}^{T}\mathbf{A}(\mathbf{s}(\tau),\tau)\mathrm{d}\tau, 
  \end{aligned}
\end{equation}
where $\Omega$ is a bounded open domain in $\mathbb{R}^d$, and $\mathbf{s}(\tau)$ satisfies $\frac{\mathrm{d}\mathbf{s}}{\mathrm{d}\tau}=\mathbf{A}(\mathbf{s},\tau)$ with $\mathbf{s}(T)=\mathbf{x}_T$. 
{\color{black}{
Given that $\mathbf{A}$ is \textbf{continuously differentiable} ($\mathbf{A} \in C^1$) and \textbf{its gradient is uniformly bounded}, 
it further follows that \textbf{the flow map \(\operatorname{D}^{\mathbf{A}}_{T \to t}\) is invertible }}} and its inverse operator is given by 
\begin{equation}
  \begin{aligned}
    \operatorname{D}^{\mathbf{A}}_{t \to T}:~\widetilde{\Omega}(t) &\rightarrow \Omega \\
    \mathbf{x}_t &\mapsto \mathbf{x}_T=\mathbf{x}_t+\int_{t}^{T}\mathbf{A}(\mathbf{s}(\tau),\tau)\mathrm{d}\tau,   
  \end{aligned}
\end{equation}
where $\mathbf{s}(\tau)$ satisfies $\frac{\mathrm{d}\mathbf{s}}{\mathrm{d}\tau}=\mathbf{A}(\mathbf{s},\tau)$ with $\mathbf{s}(t)=\mathbf{x}_t$. 
Therefore, under the action of the reversible mapping $\operatorname{D}^{\mathbf{A}}_{T \leftrightarrows t}(\cdot)$, the regions $\Omega$ and $\widetilde{\Omega}(t)$ are \textbf{diffeomorphic} ($\Omega \cong \widetilde{\Omega}(t)$). 
By the composition property of the flow mapping $\operatorname{D}^{\mathbf{A}}_{T \to t}(\cdot)$ and the transitivity of diffeomorphisms, it follows that for any $t^{\star},t^{\dagger}\in[0,T]$:
\begin{subequations}
  \begin{align}
    &\widetilde{\Omega}(t^{\star}) = \operatorname{D}^{\mathbf{A}}_{t^{\dagger} \to t^{\star}}\left(\widetilde{\Omega}(t^{\dagger})\right), \label{eq-Omega-star-dagger} \\
    &\widetilde{\Omega}(t^{\star}) \cong \widetilde{\Omega}(t^{\dagger}). 
  \end{align}  
\end{subequations}
\par \noindent
Additionally, let
\begin{align}\label{eq-widetildeQT}
  \widetilde{Q}_T &:= \left\{\widetilde{\Omega}(t):~t \in [0, T]\right\}. 
\end{align}
\par
{\color{red}{The symbols $C^1$, and $L^P$ are retained to represent their conventional mathematical meanings. 
In this paper, we denote by $H$ the Hilbert space composed of $L^2$-functions and equipped with the $L^2$-inner product. }}
Furthermore, $H_0$ denotes the Hilbert space under the \(L^2\)-inner product with compact support on a bounded domain \(\Omega\), i.e.,  
\begin{align*}
  H_0({\color{black}{\Omega}}) &:= \left\{ f \in L^2({\color{black}{\mathbb{R}^d}}) : \operatorname{supp}\{f\} \subset\subset {\color{black}{\Omega}} \right\}, \\
  \langle f, g \rangle &:= \int_{{\color{black}{\mathbb{R}^d}}} fg \, \mathrm{d}x = \int_{{\color{black}{\Omega}}} fg \, \mathrm{d}x, \quad f, g \in H_0({\color{black}{\Omega}}). 
\end{align*}
\begin{remark}[\(H_0(\Omega) \subset H(\mathbb{R}^d)\)]
  Let \(H({\color{black}{\mathbb{R}^d}})\) denote the Hilbert space formed by the \(L^2\)-space defined on the entire domain \({\color{black}{\mathbb{R}^d}}\) equipped with the \(L^2\)-inner product. 
  Let \(\mathcal{D}(\Omega)\) denote the set of all functions with compact support in \(\Omega\). Then,
  \(
  H_0({\color{black}{\Omega}}) = H({\color{black}{\mathbb{R}^d}}) \cap \mathcal{D}(\Omega).
  \)
\end{remark}
Define the following \textbf{function evolution operator} (an explicit time-dependent transformation):
\begin{align}\label{eq-psi-evolution-operator}
  \operatorname{\Phi}^{\mathbf{A}}_{T \to t}[f_T](\mathbf{x}_t):=f_T(\operatorname{D}^{\mathbf{A}}_{t \to T}(\mathbf{x}_t))=f_T(\mathbf{x}_t+\int_{t}^{T}\mathbf{A}(\mathbf{s}(\tau),\tau)\mathrm{~d}\tau), ~\mathbf{x}_t\in\widetilde{\Omega}(t),~t\in[0,T],~f_T\in L^2(\Omega),   
\end{align}
where $\mathbf{s}(\tau)$ satisfies $\frac{\mathrm{d}\mathbf{s}}{\mathrm{d}\tau}=\mathbf{A}(\mathbf{s},\tau)$ with $\mathbf{s}(t)=\mathbf{x}_t$. 
{\color{black}{It is straightforward to verify the composition property of this operator: for any $t^{\star},t^{\dagger}\in[0,T]$ it holds that }}
\begin{align}\label{eq-evolution-operator-composition-property}
  {\color{black}{\operatorname{\Phi}^{\mathbf{A}}_{T \to t^{\star}}[f_T](\mathbf{x})=\operatorname{\Phi}^{\mathbf{A}}_{t^{\dagger} \to t^{\star}} \circ \operatorname{\Phi}^{\mathbf{A}}_{T \to t^{\dagger}}[f_T](\mathbf{x}), ~\mathbf{x}\in\widetilde{\Omega}(t^{\star}). }}
\end{align}
Since $\operatorname{D}^{\mathbf{A}}_{T \to t}(\cdot)$ is invertible, 
the operator $\operatorname{\Phi}^{\mathbf{A}}_{T \to t}[\cdot]$ is then \textbf{invertible}. Denoting $f_t = \operatorname{\Phi}^{\mathbf{A}}_{T \to t}[f_T](\mathbf{x}_t)$, we have:
\begin{align}
  f_T = \operatorname{\Phi}^{\mathbf{A}}_{t \to T}[f_t](\mathbf{x}_T) = f_t\left(\mathbf{x}_T - \int_{t}^{T} \mathbf{A}(\mathbf{s}(\tau), \tau)  \mathrm{d}\tau\right), ~\mathbf{x}_T \in \Omega, 
\end{align}
where $\mathbf{s}(\tau)$ satisfies $\frac{\mathrm{d}\mathbf{s}}{\mathrm{d}\tau} = \mathbf{A}(\mathbf{s}, \tau)$ with $\mathbf{s}(T) = \mathbf{x}_T$. 
{\color{red}{
\begin{remark}[Deductive foundation: choose between ``uniformly bounded gradient'' or restricted time domain]
  In fact, the uniformly bounded gradient, $\|\nabla_{\mathbf{x}}\mathbf{A}(\mathbf{x},t)\|_{\mathrm{op}} \leq L_A$, guarantees that the characteristic curves of the adjoint equation \eqref{eq-adjoint} never intersect, i.e., $T^{*} = \infty$. 
  Consequently, $T$ is unrestricted from above, allowing us to discuss the long-time behavior. 
  Naturally, if we remove the restriction of ``uniformly bounded gradient'' and consider the possibility of intersecting characteristics ($T^* < \infty$), we would then have to state that the discussion is valid only within the time interval $[0, T]$ before the first intersection of any two characteristic curves ($T < T^*$). 
  Meanwhile, the ``linear growth condition'' must be retained. Even so, the Cauchy initial-value problem and the Cauchy final-value problem for $\frac{\mathrm{d}\mathbf{s}}{\mathrm{d}\tau} = {\mathbf{A}(\mathbf{s},\tau)}$ remain well-posed globally on $[0, T]$. 
  Hence, both $\operatorname{D}^{\mathbf{A}}_{T \to t}(\cdot)$ and $\Phi^{\mathbf{A}}_{T \to t}[\cdot]$ are still invertible. 
  However, a drawback of this approach is that $T$ might be very small, meaning our conclusion holds only for a very short time interval (\textbf{even though the solution exists globally on $[0, T]$, $T$ itself could be very small}). 
  In summary, we adopt the ``uniform boundedness of the gradient'' as a fundamental assumption for the subsequent discussion. 
\end{remark}
}}
Finally, we need to introduce \textbf{abstract functions taking values in $L^p$-spaces}, defined as follows:
\begin{definition}[Abstract function spaces of $L^{q}(1 \leq q < \infty)$ in time and $L^p(1 \leq p \leq \infty)$ in space]\label{def-Lq-t-Lp-x-FunSpace}
\begin{align*}
  L^{q}([0,T],L^p(\Omega)):=\left\{v(\cdot\ ,t)\in L^p(\Omega),\ \forall t\in[0,T]:\ \|v(\cdot\ ,t)\|_{L^p(\Omega)}\in L^{q}([0,T])\right\}. 
\end{align*}
That is, for each fixed $t \in[0, T]$, $v(\cdot, t) \in L^p(\Omega)$ is required,
and the mapping $t \mapsto\|v(\cdot~, t)\|_{L^p(\Omega)}$ is bounded in the sense of $L^q$-norm on $[0, T]$. 
The norm is defined as 
\begin{align}\label{eq-norm-Lqt-Lpx}
  \|v\|_{L^{q}([0,T],L^p(\Omega))}:=\left(\int_{0}^{T}\|v(\cdot\ ,t)\|^q_{L^p(\Omega)}\mathrm{~d}t\right)^{1/q}.
\end{align} 
\end{definition}
\par\noindent
In this paper, with respect to the spatial variable $\mathbf{x}$ on $\Omega$ under the $L^2$-inner product, the space $L^{q}([0,T],L^2(\Omega))$ can be written as $L^{q}([0,T],H(\Omega))$. 
\begin{remark}[The case $q=\infty$]
  The space $L^{\infty}([0,T],L^p(\Omega))$ is not explicitly defined through the formula \eqref{eq-norm-Lqt-Lpx}, but rather constructed via a limiting process ($q\to\infty$), 
  as detailed in Lemma \ref{lemma-Lp-Linfty}. 
\end{remark}
\par \noindent
{\color{red}{
The abstract function spaces defined in Definition \ref{def-Lq-t-Lp-x-FunSpace} are commonly referred to as \textbf{Bochner spaces}, 
sometimes denoted concisely by $L^q_t L^p_x ([0,T]\times\Omega)$. For further details on Bochner spaces, we refer the reader to \cite{ref-RRS16-3D-NS}. 
}}

\section{Comprehensive Mathematical Definition of Integral Invariant Model}\label{Sec-def-WeakSol}
\begin{definition}[Comprehensive Mathematical Formulation of Integral Invariant Model]\label{def-WeakSolution}
Let $T$ be a positive constant. 
The bounded open domain $\Omega$ is mapped via the flow mapping $\operatorname{D}^{\mathbf{A}}_{T \to t}$
from the final time $T$ to another time $t\in[0,T]$, resulting in a time-varying bounded domain $\widetilde{\Omega}(t)$, abbreviated as $\widetilde{\Omega}^{t}$. 
{\color{red}{Let $\widetilde{Q}_T := \left\{\widetilde{\Omega}(t):\ t\in[0,T]\right\}$. }}
The function $U(\mathbf{x},t)\in L^{1}([0,T],H(\widetilde{\Omega}(t)))$ is called the solution of the integral invariant model,   
if {\color{black}{for any \(\Psi(\mathbf{x}) \in H_{0}(\Omega)\) }} 
the following Cauchy initial value problem is satisfied:
\begin{subequations}\label{eq-ComprehensiveIntegralInvariantModel}
  \begin{align}
    &\frac{\mathrm{d}}{\mathrm{d}t}\int_{\widetilde{\Omega}(t)}U(\mathbf{x},t)\psi(\mathbf{x},t)\mathrm{~d}\mathbf{x} = 0, \label{eq-Def-IIM}\\ 
    &\psi(\mathbf{x},t)=\operatorname{\Phi}^{\mathbf{A}}_{T \to t}[\Psi](\mathbf{x}), \label{eq-psi-Psi} \\
    &\widetilde{\Omega}(t)=\operatorname{D}^{\mathbf{A}}_{T \to t}(\Omega), \label{eq-Def-tildeOmega} \\ 
    &{\color{black}{\mathrm{I.C.}\quad U(\mathbf{x},0)=\widetilde{U}_{0}(\mathbf{x}),\ \forall \mathbf{x}\in\widetilde{\Omega}(0). }}  \label{eq-Def-U-IC}
  \end{align} 
\end{subequations}
where $L^{1}([0,T],H(\widetilde{\Omega}(t))):=\left\{U(\cdot\ ,t)\in L^2(\widetilde{\Omega}(t)),\ \forall t\in[0,T]:\ \|U(\cdot\ ,t)\|_{L^2(\widetilde{\Omega}(t))}\in L^{1}([0,T])\right\}$. 
Here, {\color{black}{$\widetilde{U}_0\in L^2(\widetilde{\Omega}(0))$, and $\widetilde{U}_0$ is extended to the entire space $\mathbb{R}^d$ by zero extension. }}  
{\color{black}{Additionally, $U(\mathbf{x},t)$ is alse called the integral invariant-type weak solution of the linear scalar transport equation\eqref{eq-scalar-linear-transport}. }} 
\end{definition}
The system of equations \eqref{eq-ComprehensiveIntegralInvariantModel} involves three unknowns—$U$, $\psi$, and $\widetilde{\Omega}$—governed by three control equations: Eqs.\eqref{eq-Def-IIM}, \eqref{eq-psi-Psi}, and \eqref{eq-Def-tildeOmega}   
with three solution-determining conditions—$\psi(\mathbf{x},T)=\Psi(\mathbf{x})$, $\widetilde{\Omega}(T)=\Omega$ and $U(\mathbf{x},0)=\widetilde{U}_{0}(\mathbf{x})~\text{in}~\widetilde{\Omega}(0)$. 
\par
\textbf{Note that the definition of the space-time domain $\widetilde{Q}_T$ implies the following property:  
For any point $(\mathbf{x}, t) \in \widetilde{Q}_T$, the trajectory under the flow map $\operatorname{D}^{\mathbf{A}}_{t \to 0}$ 
traced backward in time will remain within $\widetilde{Q}_T$ until reaching the bottom boundary (the bounded open set $\widetilde{\Omega}(0) \subset \mathbb{R}^d$), 
without ever exiting the space-time domain $\widetilde{Q}_T$ itself. }  
This observation motivates our definition of the integral invariant model on $\widetilde{Q}_T$ in the form of a Cauchy initial value problem, 
where the initial condition $\widetilde{U}_0$ is prescribed only on the bottom boundary $\widetilde{\Omega}(0)$. 
We shall prove later that this constitutes a well-posed formulation. 
\begin{remark}
  In Definition \ref{def-WeakSolution}, the prescribed solution space $L^{2}([0,T],H(\widetilde{\Omega}(t)))$ not only satisfies the definition of Bochner space in Def.\ref{def-Lq-t-Lp-x-FunSpace}, but also requires special attention that the integration domain of its norm is time-dependent. Specifically,
  \begin{align*}
    \|U(\cdot~,t)\|_{L^2(\widetilde{\Omega}(t))} &:= \left( \int_{\widetilde{\Omega}(t)} |U(\mathbf{x},t)|^2  \mathrm{~d}\mathbf{x} \right)^{1/2}, ~\forall t \in [0,T],  \\
    \|U(\mathbf{x},t)\|_{L^{2}([0,T],H(\widetilde{\Omega}(t)))} &:= \int_{0}^{T}\|U(\cdot~,t)\|_{L^2(\widetilde{\Omega}(t))}\mathrm{~d}t. 
  \end{align*}
\end{remark}
{\color{red}{
\begin{remark}[The definition of the integral invariant type weak solution \ref{def-WeakSolution} avoids the Cauchy terminal value problem of the adjoint equation \eqref{eq-adjoint}, thereby reducing the regularity requirement of the test function $\Psi$ from $C^1$ to $L^2$]\label{remark-psi-Cauchy}
  We use ``continuously differentiable'' and ``linear growth'' to ensure that $\frac{\mathrm{d}\mathbf{s}}{\mathrm{d}\tau}=\mathbf{A}(\mathbf{s},\tau)$ is invertible. 
  Using the method of characteristics, it follows that the Cauchy initial/terminal value problems for the adjoint equation \eqref{eq-adjoint} are also well-posed. 
  However, solving the adjoint equation via characteristics has a prerequisite: it requires that $\Psi \in C^1(\Omega)$, so that all derivatives in $\frac{\mathrm{d}\psi}{\mathrm{d}t}=\partial_t \psi + \nabla_{\mathbf{x}}\psi\cdot\frac{\mathrm{d}\mathbf{x}}{\mathrm{d}t}$ are well-defined.  
  Nevertheless, we note that in the traditional Galerkin variational forms for elliptic/parabolic/hyperbolic equations, test functions are often set in weaker Sobolev spaces. 
  To make our work more consistent with the common understanding of weak solutions, it is necessary to find a way to select test functions with lower regularity.
  Therefore, in Definition~\ref{def-WeakSolution}, the integral invariant model~\eqref{eq-ComprehensiveIntegralInvariantModel} does not adopt the adjoint equation~\eqref{eq-adjoint} as the governing equation for the time-varying test function $\psi(\mathbf{x},t)$, 
  but instead directly employs the operator $\Phi^{\mathbf{A}}_{T \to t}[\cdot]$ to evolve $\psi$.  
  This explicit construction method characterizes $\psi$ at any arbitrary time, thereby allowing us to lower the regularity requirement for $\Psi$ to $L^2$ based on the needs of subsequent lemmas and theorems.  
  (Indeed, if we set $\Psi \in C^1(\Omega)$, then in the subsequent stability proof of the solution, the solution obtained from the existence proof only possesses $L^2$ regularity with respect to the spatial variable $x$, making it impossible to construct a test function that meets the required specifications. However, when $\Psi \in L^2(\Omega)$, the $L^2$-solution naturally allows for the construction of an $L^2$-test function, see Eq.\eqref{eq-Stability-testFun-Psi} in the following Subsection \ref{SubSec-Stability-Uniqueness}.) 
  This treatment---where ``the test function belongs to the low regularity $L^2$ space''---more naturally aligns with the conventional framework of finite element weak solutions and better reflects the fundamental feature of ``the low regularity of weak solutions''. 
  \par
  Finally, we emphasize that numerical algorithms based on discretizing Equation \eqref{eq-IntegralInvariantFormulation} universally employ a fixed background grid $\Omega$ and predefined test function $\Psi$ at the time layer $t^{n+1}$.  
  These are then numerically inverted to $t^n$ to obtain approximations $\widetilde{\Omega}^{n+1,n} \approx \widetilde{\Omega}(t^n)$ and $\psi^{n+1,n} \approx \psi(\mathbf{x},t^n)$.  
  \textbf{To remain compatible with the existing numerical algorithmic framework, } the definition of the integral invariant model \eqref{eq-ComprehensiveIntegralInvariantModel} naturally adopts a \textbf{time-reversed description} for the evolution of $\widetilde{\Omega}(t)$ and $\psi(\mathbf{x},t)$.  
  This necessitates Cauchy terminal value problems—both for the characteristic equation $\frac{\mathrm{d}\mathbf{s}}{\mathrm{d}\tau} = \mathbf{A}(\mathbf{s},\tau)$ and the adjoint equation $\psi_t + \mathbf{A} \cdot \nabla \psi = 0$.  
  The operator $\operatorname{\Phi}^{\mathbf{A}}_{T \to t}[\cdot]$ circumvents the well-posedness challenges of solving the adjoint equation directly—an astute and pragmatic design choice. 
  \par
  In practice, during numerical computation, they do not employ a PDE solver to directly solve the adjoint equation.  
  Instead, they directly leverage the Formula \eqref{eq-psi-evolution-operator} in this paper to approximate $\psi(\mathbf{x},t^n)$.  
  Therefore, our definition of integral invariant model inherently preserves the background information from discrete algorithms. 
\end{remark}
}}
\par
For arbitrary $t^{\star},t^{\dagger}\in[0,T]$, let
\begin{align*}
  &\widetilde{\Omega}^{\star}:=\widetilde{\Omega}(t^{\star}),\quad \widetilde{\Omega}^{\dagger}:=\widetilde{\Omega}(t^{\dagger}); \\
  &\psi^{\star}(\mathbf{x}):=\psi(\mathbf{x},t^{\star}), \quad \mathbf{x}\in\widetilde{\Omega}^{\star}; \\
  &\psi^{\dagger}(\mathbf{x}):=\psi(\mathbf{x},t^{\dagger}), \quad \mathbf{x}\in\widetilde{\Omega}^{\dagger}.
\end{align*}
According to the composition properties and reversibility of $\operatorname{D}^{\mathbf{A}}_{T \to t}(\cdot)$ and $\operatorname{\Phi}^{\mathbf{A}}_{T \to t}[\cdot]$, the following relation can be readily obtained:
\begin{align}\label{eq-psi-t1-psi-t2}
  \psi^{\star}(\mathbf{x})=\psi^{\dagger}\left(\mathbf{x}+\int_{t^{\star}}^{t^{\dagger}}\mathbf{A}(\mathbf{s}(\tau),\tau)\mathrm{~d}\tau\right),~\forall \mathbf{x}\in\widetilde{\Omega}^{\star}. 
\end{align}
\par
\textbf{Several equivalent forms of Eq.\eqref{eq-Def-IIM} in Definition \ref{def-WeakSolution}  are presented below to facilitate subsequent proofs: }
\begin{subequations}
  \begin{align}
    &\frac{\mathrm{d}}{\mathrm{d}t}\int_{\widetilde{\Omega}(t)}U(\mathbf{x},t)\psi(\mathbf{x},t)\mathrm{~d}\mathbf{x} = 0 \notag \\
    &\Longleftrightarrow
    \int_{\widetilde{\Omega}(t_2)} U(\mathbf{x},t_2)\psi(\mathbf{x},t_2) \mathrm{~d}\mathbf{x} = \int_{\widetilde{\Omega}(t_1)} U(\mathbf{x},t_1)\psi(\mathbf{x},t_1) \mathrm{~d}\mathbf{x},\ \forall t_1,t_2 \in [0,T]; \\
    &\Longleftrightarrow
    \int_{\widetilde{\Omega}(t)} U(\mathbf{x},t)\psi(\mathbf{x},t) \mathrm{~d}\mathbf{x} = \int_{\widetilde{\Omega}(0)} \widetilde{U}_{0}(\mathbf{x})\psi(\mathbf{x},0) \mathrm{~d}\mathbf{x},\ \forall t \in [0,T]. \label{eq-WeakSol-EquivalentForm2}
  \end{align}
\end{subequations}

\section{Some {\color{red}{Useful}} Lemmas}\label{sec-lemmas}
Note that in the following lemmas, we maintain the assumptions that $\mathbf{A}(\mathbf{x},t)$ has \textbf{continuous first-order partial derivatives} and satisfies the \textbf{uniformly bounded gradient condition} \eqref{eq-gardA-UniformBoundedness}. 
Additionally, please pay attention to the \textbf{uniformly bounded divergence} as described in Eq.\eqref{eq-div-UniformBoundedness}: 
\begin{align*}
  |\nabla_{\mathbf{x}} \cdot \mathbf{A}(\mathbf{x}, t)| \leq M_A, ~\forall \mathbf{x} \in \mathbb{R}^d, t \in\left[0, T\right]. 
\end{align*}
\par
\begin{lemma}[Support of the Time-Dependent Test Function \(\psi\)]\label{Lemma-psi-supp}
  Let \(\Omega \subset \mathbb{R}^d\) be a bounded domain, and assume \(\operatorname{supp}\{\Psi(x)\} \subset\subset \Omega\). Let \(\psi(\mathbf{x}, t)\) satisfy
  $$
    \psi(\mathbf{x},t)=\operatorname{\Phi}^{\mathbf{A}}_{T \to t}[\Psi](\mathbf{x}),~\forall t\in[0,T]. 
  $$
  Then, {\color{red}{for fixed $t$,}} it holds that
  \[
  {\color{red}{\operatorname{supp}\{\psi(\cdot, t)\}}} = \overline{\left\{\mathbf{x}: ~\mathbf{x} + \int_{t}^{T} \mathbf{A}(\mathbf{s}(\tau), \tau) \, \mathrm{d}\tau \in \operatorname{supp}\{\Psi\}\right\}}. 
  \]
  Let \(\widetilde{\Omega}^t = \operatorname{D}^{\mathbf{A}}_{T \rightarrow t}(\Omega)\). Then, \({\color{red}{\operatorname{supp}\{\psi(\cdot, t)\}}} \subset\subset \widetilde{\Omega}^t\).
\end{lemma}

\begin{proof} 
  {\color{red}{The proof is straightforward. }}
\end{proof}

\begin{remark}
  $\Psi$ depends only on $\mathbf{x}$ and is independent of time $t$. 
  By treating $t$ as a fixed time instant, we fix $\psi(\mathbf{x},t)$ to be a function of $\mathbf{x}$ alone, denoted as $\psi(\mathbf{x};t)$. 
  {\color{red}{Consequently, $\operatorname{supp}\{\psi(\mathbf{x};t)\}$ becomes a set, while being defined with respect to \(\mathbf{x}\), is contingent upon the chosen value of the parameter $t$.  }}
\end{remark}  
Lemma \ref{Lemma-psi-supp} establishes that the time-dependent test function $\psi$ maintains \textbf{compact support} throughout its evolution, and the evolution of its support precisely aligns with the movement and deformation of $\widetilde{\Omega}(t)$.
\par
{\color{red}{
\begin{lemma}[$L^P(1 \leq P < \infty)$-Norm Control of Time-Dependent Test Functions $\psi$ on Moving Deforming Domains $\widetilde{\Omega}^t$ or Whole Space $\mathbb{R}^d$ Across Time]\label{Lemma-psi-Omega-Rd-LP}
Let $1 \leq p < \infty$, and let $\Psi \in L^p(\mathcal{G}_T)$, where the domain $\mathcal{G}_T$ is either a bounded domain $\Omega \subset \mathbb{R}^d$ or the entire space $\mathbb{R}^d$. For any $t \in [0, T]$, define the time-evolved domain $\mathcal{G}_t$ and the function $\psi(\cdot, t)$ via the flow mapping and evolution operator respectively:
\[
\mathcal{G}_t = \mathrm{D}_{T \to t}^{\mathbf{A}}(\mathcal{G}_T); \quad \psi(\mathbf{x}, t) = \Phi_{T \to t}^{\mathbf{A}}[\Psi](\mathbf{x}), ~\mathbf{x} \in \mathcal{G}_t.
\]
Then, the following norm equivalence holds uniformly in time:
\begin{align*}
e^{-M_A |T-t|/p} \, \|\Psi\|_{L^p(\mathcal{G}_T)} \leq \|\psi(\cdot, t)\|_{L^p(\mathcal{G}_t)} \leq e^{M_A |T-t|/p} \, \|\Psi\|_{L^p(\mathcal{G}_T)}, \quad \forall t \in [0, T],
\end{align*}
where the constant $M_A$ is the uniform bound on the divergence of $\mathbf{A}$ as given in the assumption \eqref{eq-div-UniformBoundedness}.
\end{lemma}
}}


\begin{proof}
  It is clear that 
  \[
    \|\psi(\cdot\ , t)\|^P_{L^P(\mathcal{G}_t)} 
    = \int_{\mathcal{G}_t} \left|\psi(\mathbf{x}_t, t)\right|^P \, \mathrm{d}\mathbf{x}_t
    =\int_{\mathcal{G}_t} \left|\operatorname{\Phi}^{\mathbf{A}}_{T \to t}[\Psi](\mathbf{x}_t)\right|^P \, \mathrm{d}\mathbf{x}_t
    =\int_{\mathcal{G}_t} \left|\Psi(\operatorname{D}^{\mathbf{A}}_{t \to T}(\mathbf{x}_t))\right|^P \, \mathrm{d}\mathbf{x}_t. 
  \]
  Perform the following change of variables:
  {\color{black}{
  \[
    \mathbf{x}_T := \operatorname{D}^{\mathbf{A}}_{t \to T}(\mathbf{x}_t)
  \]
  }}
  with the corresponding transformations:
  {\color{black}{
  \[
  \begin{aligned}
    \mathcal{G}_t \mapsto \mathcal{G}_T.  
  \end{aligned}
  \]
  }}
  Given that the flow map $\operatorname{D}^{\mathbf{A}}_{t \to T}$ is invertible with inverse $\operatorname{D}^{\mathbf{A}}_{T \to t}$, we have:
  \begin{align*}
    \mathbf{x}_t &= \left(\operatorname{D}^{\mathbf{A}}_{t \to T}\right)^{-1}(\mathbf{x}_{T}) = \operatorname{D}^{\mathbf{A}}_{T \to t}(\mathbf{x}_T), \\
    \mathrm{d}\mathbf{x}_t &= \left| \det J_{\operatorname{D}^{\mathbf{A}}_{T \to t}}(\mathbf{x}_T) \right| \mathrm{d}\mathbf{x}_T,
  \end{align*}
  where $J_{\operatorname{D}^{\mathbf{A}}_{T \to t}}(\mathbf{x}_T)$ is the Jacobian matrix of the flow map $\operatorname{D}^{\mathbf{A}}_{T \to t}$. 
  Then 
  \[
    \|\psi(\cdot\ , t)\|^P_{L^P(\mathcal{G}_t)} 
    =\int_{\mathcal{G}_T} \left|\Psi(\mathbf{x}_T)\right|^P \left| \det J_{\operatorname{D}^{\mathbf{A}}_{T \to t}}(\mathbf{x}_T) \right| \mathrm{d}\mathbf{x}_T. 
  \]
  Noting the definition (Eq.\eqref{eq-FlowMap-tildeOmega}) of the reversible flow map \(\operatorname{D}^{\mathbf{A}}_{T \to t}\), an application of \textbf{Liouville's Formula} \cite{ref-A92} yields
  \begin{align*}
    \left|\operatorname{det} J_{\operatorname{D}^{\mathbf{A}}_{T \to t}}(\mathbf{x}_T)\right|=\left|\operatorname{det} J_{\operatorname{D}^{\mathbf{A}}_{t \to T}}(\mathbf{x}_t)\right|^{-1}=\exp \left(\int_{t}^{T} \nabla \cdot \mathbf{A}(\mathbf{s}(\tau), \tau) d \tau\right)^{-1}=\exp \left(-\int_{t}^{T} \nabla \cdot \mathbf{A}(\mathbf{s}(\tau), \tau) d \tau\right).  
  \end{align*}
  Utilizing the \textbf{uniformly bounded divergence condition} (Eq. \eqref{eq-div-UniformBoundedness}), we obtain the following estimate:
  \begin{align*}
    \exp(-M_A|T-t|) \leq \left|\det J_{\operatorname{D}^{\mathbf{A}}_{T \to t}}(\mathbf{x}_T)\right| \leq \exp(M_A|T-t|). 
  \end{align*}
  Consequently,
  \[
    e^{-M_A|T-t|} \int_{\mathcal{G}_T} \left|\Psi(\mathbf{x}_T)\right|^P \mathrm{d}\mathbf{x}_T \leq \int_{\mathcal{G}_T} \left|\Psi(\mathbf{x}_T)\right|^P \left| \det J_{\operatorname{D}^{\mathbf{A}}_{T \to t}}(\mathbf{x}_T) \right| \mathrm{d}\mathbf{x}_T \leq  e^{M_A|T-t|} \int_{\mathcal{G}_T} \left|\Psi(\mathbf{x}_T)\right|^P \mathrm{d}\mathbf{x}_T, 
  \]
  and thus
  \[
    e^{-M_A|T-t|/P} \|\Psi(\cdot~;T)\|_{L^P(\mathcal{G}_T)} \leq \|\psi(\cdot\ , t)\|_{L^P(\mathcal{G}_t)} \leq e^{M_A|T-t|/P} \|\Psi(\cdot~;T)\|_{L^P(\mathcal{G}_T)}, ~1 \leq P < \infty.
  \]
\end{proof}

{\color{red}{
\begin{remark}
  When $\mathcal{G}_T=\Omega\subset\mathbb{R}^d$, $D^{\mathbf{A}}_{T \to t}(\mathcal{G}_T)=\widetilde{\Omega}(t)$;
  When $\mathcal{G}_T=\Omega\subset\mathbb{R}^d$, $D^{\mathbf{A}}_{T \to t}(\mathcal{G}_T)=\mathbb{R}^d$. 
\end{remark}
}}
When $\nabla \cdot \mathbf{A} = 0$(incompressible flow), $M_A$ diminishes to $0$ and the inequalities in Lemma \ref{Lemma-psi-Omega-Rd-LP} become strict equalities.  
Note that in Lemma \ref{Lemma-psi-Omega-Rd-LP}, the control factor $e^{\pm M_A|T-t|/P}$ is \textbf{global} and \textbf{independent of $\mathcal{G}_t$}, 
ensuring that the inequalities remain valid even when considering the entire space $\mathbb{R}^d$ with infinite Lebesgue measure. 
In fact, this factor can be conservatively bounded by $e^{\pm M_A T / P}$, making the control relationship \textbf{time-independent}. 
This lemma form the foundation for key estimates in subsequent proofs. 

\begin{lemma}[Bounded Linear Functional on \(H(\mathbb{R}^d)\)]\label{Lemma-f-H*}
  Assume \(\mathfrak{U} \in L^2(\mathbb{R}^d)\), \(t_p\in[0,T)\), \(\Delta t > 0\) and satisfies \(t_p + \Delta t \in (0,T]\). For any \(W \in H(\mathbb{R}^d)\), define the functional \(f\) as follows:
  \[
  \begin{aligned}
  & f: \ {\color{black}{H(\mathbb{R}^d) \longrightarrow \mathbb{R}}}, \\
  & {\color{red}{f(W; \mathfrak{U}, \mathbf{A}, t_p, \Delta t) = \int_{\mathbb{R}^d} \mathfrak{U}  V(\mathbf{x}; W, \mathbf{A}, t_p, \Delta t) \, \mathrm{d}\mathbf{x}, }}\\
  & V(\mathbf{x}; W, \mathbf{A}, t_p, \Delta t) = W\left(\mathbf{x} + \int_{t_p}^{t_p + \Delta t} \mathbf{A}(\mathbf{s}(\tau), \tau) \, \mathrm{d}\tau\right), 
  \end{aligned}
  \]
  where $\frac{\mathrm{d}\mathbf{s}}{\mathrm{d}\tau}=\mathbf{A}(\mathbf{s},\tau)$ with $\mathbf{s}(t_p)=\mathbf{x}$.  
  Then, \(\boxed{f \in H^{\prime}(\mathbb{R}^d)}\), i.e., \(f\) is a \uwave{bounded} \underline{linear} operator on \(H(\mathbb{R}^d)\). 
\end{lemma}

\begin{proof}\hfill\par
  \begin{itemize}
    \item \textbf{Linearity}: For all \(\phi, \varphi \in H(\mathbb{R}^d)\) and \(\lambda, \mu \in \mathbb{R}\),
    {\footnotesize
    \begin{align*}
      f(\lambda \phi + \mu \varphi; \mathfrak{U}, \mathbf{A}, t_p, \Delta t)
      &= \int_{\mathbb{R}^d} \mathfrak{U}  V(\mathbf{x}; \lambda \phi + \mu \varphi, \mathbf{A}, t_p, \Delta t) \, \mathrm{d}\mathbf{x} \\
      &= \int_{\mathbb{R}^d} \mathfrak{U}  (\lambda \phi + \mu \varphi)\left(\mathbf{x} + \int_{t_p}^{t_p + \Delta t} \mathbf{A}(\mathbf{s}(\tau), \tau) \, \mathrm{d}\tau\right) \, \mathrm{d}\mathbf{x} \\
      (H(\mathbb{R}^d)\text{ is a linear space}) &= \int_{\mathbb{R}^d} \mathfrak{U} \left(\lambda \phi \left(\mathbf{x} + \int_{t_p}^{t_p + \Delta t} \mathbf{A}(\mathbf{s}(\tau), \tau) \, \mathrm{d}\tau\right) + \mu \varphi\left(\mathbf{x} + \int_{t_p}^{t_p + \Delta t} \mathbf{A}(\mathbf{s}(\tau), \tau) \, \mathrm{d}\tau\right)\right) \, \mathrm{d}\mathbf{x} \\
      &= \lambda \int_{\mathbb{R}^d} \mathfrak{U} \phi \left(\mathbf{x} + \int_{t_p}^{t_p + \Delta t} \mathbf{A}(\mathbf{s}(\tau), \tau) \, \mathrm{d}\tau\right) \, \mathrm{d}\mathbf{x}
       + \mu \int_{\mathbb{R}^d} \mathfrak{U} \varphi\left(\mathbf{x} + \int_{t_p}^{t_p + \Delta t} \mathbf{A}(\mathbf{s}(\tau), \tau) \, \mathrm{d}\tau\right) \, \mathrm{d}\mathbf{x} \\
      &= \lambda f(\phi; \mathfrak{U}, \mathbf{A}, t_p, \Delta t) + \mu f(\varphi; \mathfrak{U}, \mathbf{A}, t_p, \Delta t).
    \end{align*}
    }
    \item \textbf{Boundedness (Continuity)}: For all \(W \in H(\mathbb{R}^d)\),
    \begin{align*}
      \left|f(W; \mathfrak{U}, \mathbf{A}, t_p, \Delta t)\right| 
      &\leq \int_{\mathbb{R}^d} \left|\mathfrak{U}  V(\mathbf{x}; W, \mathbf{A}, t_p, \Delta t)\right| \, \mathrm{d}\mathbf{x} \\
      &\leq \|\mathfrak{U}\|_{L^2(\mathbb{R}^d)} \cdot \|V(\mathbf{x}; W, \mathbf{A}, t_p, \Delta t)\|_{L^2(\mathbb{R}^d)} \\
      (\text{Lemma }\ref{Lemma-psi-Omega-Rd-LP}) &\leq \|\mathfrak{U}\|_{L^2(\mathbb{R}^d)} \cdot e^{M_A \Delta t} \|W\|_{L^2(\mathbb{R}^d)} \\
      &\leq e^{M_A T}\|\mathfrak{U}\|_{L^2(\mathbb{R}^d)} \cdot \|W\|_{L^2(\mathbb{R}^d)}. 
    \end{align*}
  \end{itemize}  
\end{proof}
\par \noindent
Lemma \ref{Lemma-f-H*} establishes the foundation for applying the Riesz Representation Theorem in the subsequent existence proof.
\par 
\hfill \par
Finally, we introduce a lemma established in previous work for enhancing $L^p$-integrability, with its proof available in \cite{ref-CHenSX18-lemma}: 
\begin{lemma}[Uniformly bounded in all $L^p$-norms - such functions must actually be essentially bounded]\label{lemma-Lp-Linfty}
  If $u \in L^p(\Omega)$ and $\|u\|_{L^p(\Omega)} \leq C$ holds for all $1 \leq p < \infty$, 
  then $u \in L^{\infty}(\Omega)$ and $\|u\|_{L^p(\Omega)} \xrightarrow{p \to \infty} \|u\|_{L^{\infty}(\Omega)}$.
\end{lemma}
\par \noindent
This lemma will be applied in Section \ref{sec-WeakSol-regularity-t-Linfty} to enhance the temporal integrability of solutions to the integral invariant model from $L^1$ to $L^{\infty}$.

\section{Well-posedness of Integral Invariant Model}\label{sec-WeakSol-Existence-Uniqueness-Stability}
\subsection{Existence of The Solution to Integral Invariant Model\eqref{eq-ComprehensiveIntegralInvariantModel}}
\begin{theorem}[Existence of The Solution to The Integral Invariant Model\eqref{eq-ComprehensiveIntegralInvariantModel}]\label{Thm-WeakSolution-Existence}
  There exists \(U(\mathbf{x}, t) \in L^{1}([0, T], H(\widetilde{\Omega}(t)))\) that satisfies the integral invariant model\eqref{eq-ComprehensiveIntegralInvariantModel}.  
\end{theorem}
\begin{proof}
{\color{red}{
Based on the equivalent definition \eqref{eq-WeakSol-EquivalentForm2} of the definition \ref{def-WeakSolution}, it is sufficient to show  
\begin{align}\label{eq-InertialIntegralInvariant-t*-t0}
    \underbrace{\int_{\widetilde{\Omega}^{*}} U(\mathbf{x},t^*)\psi^{*} \mathrm{~d}\mathbf{x}}_{LHS_{*}} = \underbrace{\int_{\widetilde{\Omega}^{0}} \widetilde{U}_{0}(\mathbf{x})\psi^{0} \mathrm{~d}\mathbf{x}}_{RHS_{0}}, ~\forall t^* \in [0,T]. 
\end{align} 
}}
It should be noted that  
{\color{black}{$\widetilde{U}_{0}\in L^2(\widetilde{\Omega}^0)$ and is defined on $\mathbb{R}^d$ with compact support 
($\operatorname{supp}\{\widetilde{U}_{0}\}\subset\subset\widetilde{\Omega}_0$). }}  
\begin{itemize}
  \item \textbf{Step1.\ For any fixed time $t^*\in(0,T]$, the function $U^*(\mathbf{x})$ is constructed via the Riesz representation theorem to satisfy equation \eqref{eq-WeakSol-EquivalentForm2}. }
    \par
    For $\Psi\in H_0(\Omega)$, by combining equation \eqref{eq-psi-Psi} with Lemmas \ref{Lemma-psi-supp} and \ref{Lemma-psi-Omega-Rd-LP}, it is shown that
    \begin{align}
      \psi^{0}\in H_0(\widetilde{\Omega}^{0}), ~\psi^{*}\in H_0(\widetilde{\Omega}^{*}).
    \end{align}
    It should be emphasized that $\Psi$, $\psi^{0}$, and $\psi^{*}$ are defined on the whole space $\mathbb{R}^d$ in a compactly supported manner (via zero extension), namely
    \begin{align}
      \Psi\in H_0(\Omega)\subset H(\mathbb{R}^d), ~\psi^{0}\in H_0(\widetilde{\Omega}^{0})\subset H(\mathbb{R}^d), ~\psi^{*}\in H_0(\widetilde{\Omega}^{*})\subset H(\mathbb{R}^d). 
    \end{align}
    From equation \eqref{eq-psi-t1-psi-t2}, it is obtained that
    \begin{align}
        \psi^0=\psi^*\left(\mathbf{x}+\int_{0}^{t^*}\mathbf{A}(\mathbf{s}(\tau),\tau)\mathrm{~d}\tau\right), ~\mathbf{x}\in\widetilde{\Omega}^{0}, 
    \end{align} 
    which implies
    \begin{align}
      RHS_{0}
      =\int_{\widetilde{\Omega}^{0}} \widetilde{U}_{0}\psi^{*}\left(\mathbf{x}+\int_{0}^{t^*}\mathbf{A}(\mathbf{s}(\tau),\tau)\mathrm{~d}\tau\right) \mathrm{~d}\mathbf{x}
      =\int_{\widetilde{\Omega}^{0}} \widetilde{U}_{0}\psi^{0}(\mathbf{x};\psi^{*},\mathbf{A},0,t^*) \mathrm{~d}\mathbf{x}. 
    \end{align}
    Therefore, the following functional can be introduced: for all $w \in H(\mathbb{R}^d)$, 
    \begin{align}\label{eq-f0(omega)}
      f^0(\omega;\widetilde{U}_{0},\mathbf{A},0,t^*)
      :=\int_{{\color{black}{\mathbb{R}^d}}} \widetilde{U}_{0}\psi^{0}(\mathbf{x};\omega,\mathbf{A},0,t^*) \mathrm{~d}\mathbf{x}
      =\int_{{\color{black}{\mathbb{R}^d}}} \widetilde{U}_{0}\omega\left(\mathbf{x}+\int_{0}^{t^*}\mathbf{A}(\mathbf{s}(\tau),\tau)\mathrm{~d}\tau\right) \mathrm{~d}\mathbf{x}.  
    \end{align}
    According to Lemma \ref{Lemma-f-H*}, the functional $f^0$ is confirmed to belong to $H^{\prime}(\mathbb{R}^d)$. The {\color{black}{Riesz representation theorem}} is therefore applicable, guaranteeing the existence of a unique $\mathcal{W}^{*}{\color{black}{\in H(\mathbb{R}^d)}}$ satisfying
    \begin{align}\label{eq-Riesz-W*-exist}
      \langle \mathcal{W}^{*},w \rangle_{L^2(\mathbb{R}^d)} = f^0(w;\widetilde{U}_{0},\mathbf{A},0,t^*), 
    \end{align}
    for all $w \in H(\mathbb{R}^d)$, with the norm identity
    \begin{align}\label{eq-Riesz-W*-norm}
      \|{\color{red}{\mathcal{W}^{*}}}\|_{L^2(\mathbb{R}^d)}=\|f^0\|_{H^{\prime}(\mathbb{R}^d)}.
    \end{align}
    \uwave{Given that $\psi^*\in H_0({\color{black}{\widetilde{\Omega}^*}})\subset H(\mathbb{R}^d)$}, the following equality is valid:
    \begin{align}\label{eq-Riesz-W*-psi*}
      \langle \mathcal{W}^{*},\psi^* \rangle_{L^2(\mathbb{R}^d)} = f^0(\psi^*;\widetilde{U}_{0},\mathbf{A},0,t^*), ~\forall \psi^*\in H_0({\color{black}{\widetilde{\Omega}^*}}).
    \end{align}
    An important observation is made that $\psi^*$ is uniquely determined by $\Psi$ via the transformation $\psi^*(\mathbf{x})=\operatorname{\Phi}^{\mathbf{A}}_{T \to t^*}[\Psi](\mathbf{x})$. The arbitrariness of $\Psi\in H_0(\Omega)\subset H(\mathbb{R}^d)$ implies that of $\psi^*$. Consequently, 
    \begin{align}
      \exists ! \mathcal{W}^{*}{\color{black}{\in H(\mathbb{R}^d)}}\ s.t.\ 
      \langle \mathcal{W}^{*},\psi^* \rangle_{L^2(\mathbb{R}^d)} = f^0(\psi^*;\widetilde{U}_{0},\mathbf{A},0,t^*), ~\forall \Psi \in H_0(\Omega).
    \end{align}
    It is observed that since $\operatorname{supp}\{\psi^*\}\subset\subset{\color{black}{\widetilde{\Omega}^{*}}}$, the following equality holds:
    \begin{align}
      \langle \mathcal{W}^{*},\psi^{*} \rangle_{L^2(\mathbb{R}^d)} = \langle \mathcal{W}^{*},\psi^{*} \rangle_{L^2({\color{black}{\widetilde{\Omega}^{*}}})}=\int_{{\color{black}{\widetilde{\Omega}^{*}}}} \mathcal{W}^{*}\psi^* \mathrm{~d}\mathbf{x};
    \end{align}  
    Similarly, since $\operatorname{supp}\{\psi^0\}\subset\subset{\color{black}{\widetilde{\Omega}^{0}}}$, we have
    \begin{align}
      f^0(\psi^*;\widetilde{U}_{0},\mathbf{A},0,t^*)=\int_{\mathbb{R}^d} \widetilde{U}_{0}\psi^{0}(\mathbf{x};\psi^{*},\mathbf{A},0,t^*) \mathrm{~d}\mathbf{x}=\int_{\widetilde{\Omega}^{0}} \widetilde{U}_{0}\psi^{0}(\mathbf{x};\psi^{*},\mathbf{A},0,t^*) \mathrm{~d}\mathbf{x}. 
    \end{align} 
    Consequently, 
    \begin{align}
    \exists ! \mathcal{W}^{*}{\color{black}{\in H(\mathbb{R}^d)}}\ \text{such that}\ 
    \int_{\color{black}{\widetilde{\Omega}^{*}}} \mathcal{W}^{*}\psi^* \mathrm{~d}\mathbf{x}=\int_{\widetilde{\Omega}^{0}} \widetilde{U}_{0}\psi^{0}(\mathbf{x};\psi^{*},\mathbf{A},0,{\color{red}{t^*}}) \mathrm{~d}\mathbf{x},\ 
    \forall \Psi \in H_0(\Omega).
    \end{align}
    The solution at time $t^*$ is consequently defined as
    \begin{align}\label{eq-Riesz-U*-W*}
      U(\mathbf{x},t^*)=\mathcal{W}^{*}(\mathbf{x})\cdot\chi_{\widetilde{\Omega}^{*}}(\mathbf{x}){\color{black}{\in H({\color{black}{\widetilde{\Omega}^{*}}})}}.
    \end{align}
    That is, by using the indicator function $\chi_{\widetilde{\Omega}^{*}}$, $\mathcal{W}^{*}$ defined on the entire space is directly restricted to the local bounded domain $\widetilde{\Omega}^{*}$. 
    Denote $U(\mathbf{x},t^*)$ simply as $U^*$ and it still satisfies Equation \eqref{eq-WeakSol-EquivalentForm2}.

  \item \textbf{Step2.\ The mapping $t^* \mapsto U^*$ defines an abstract function $\mathcal{U}(t)$, through which $U^*(\mathbf{x})$ is extended to the full solution $U(\mathbf{x},t)$. }
    \par
    The arbitrariness of $t^*$ combined with the uniqueness of the representing element $\mathcal{W}^*$ guaranteed by the Riesz representation theorem (and consequently the uniqueness of $U^*$) leads to the construction of a single-valued mapping
    \begin{align}
      \mathcal{U}:~ (0,T] \ni t^*\mapsto \mathcal{U}(t^*):=U^*\in H(\widetilde{\Omega}^{*}).  
    \end{align}
    This mapping $\mathcal{U}(\lambda)$ is referred to as an abstract function defined on $(0,T]$ taking values in the Banach space $H(\widetilde{\Omega}^{\lambda})$. 
    It should be noted that for all $\lambda\in(0,T]$, the following inequality obviously holds:
    \begin{align}
      \|\mathcal{U}(\lambda)\|_{L^2(\widetilde{\Omega}^{\lambda})}\leq\|\mathcal{W}^{\lambda}\|_{L^2(\mathbb{R}^d)}<\infty.
    \end{align}
    In fact, the mapping from $\lambda$ to $\|\mathcal{U}(\lambda)\|_{L^2(\widetilde{\Omega}^{\lambda})}$ is also established as
    \begin{equation}
      \begin{aligned}
        (0,T] &\rightarrow [0,+\infty), \\
        \lambda &\mapsto \|\mathcal{U}(\lambda)\|_{L^2(\widetilde{\Omega}^{\lambda})}. 
      \end{aligned}
    \end{equation}
  
  \item \textbf{Step3.\ {\color{red}{Determine}} the regularity of the temporal variable $t$. }
    \par
    Further regularity estimates for $\|\mathcal{U}(\lambda)\|_{L^2(\widetilde{\Omega}^{\lambda})}$ in the temporal variable are derived by revisiting equations \eqref{eq-f0(omega)}, \eqref{eq-Riesz-W*-norm}, and \eqref{eq-Riesz-U*-W*}. 
    The operator norm characterization yields
    \begin{align}
      \|f^0\|_{H^{\prime}(\mathbb{R}^d)}=\sup\left\{\left|\int_{\mathbb{R}^d} \widetilde{U}_{0}\omega\left(\mathbf{x}+\int_{0}^{t^{*}}\mathbf{A}(\mathbf{s}(\tau),\tau)\mathrm{~d}\tau\right) \mathrm{~d}\mathbf{x}\right|:\ \omega\in H(\mathbb{R}^d),\ \|\omega\|_{L^2(\mathbb{R}^d)}=1\right\}.
    \end{align}
    A key observation is made that the uniform bound
    \begin{align}
      \left|\int_{\mathbb{R}^d} \widetilde{U}_{0}\omega\left(\mathbf{x}+\int_{0}^{t^{*}}\mathbf{A}(\mathbf{s}(\tau),\tau)\mathrm{~d}\tau\right) \mathrm{~d}\mathbf{x}\right|
      &\leq
      \|\widetilde{U}_{0}\|_{L^2(\mathbb{R}^d)}\cdot e^{M_A t^* {\color{black}{/2}}} \|\omega\|_{L^2(\mathbb{R}^d)} 
      =e^{M_A t^* {\color{black}{/2}}}\|\widetilde{U}_{0}\|_{L^2(\mathbb{R}^d)} \notag \\
      &=e^{M_A t^* {\color{black}{/2}}}\|\widetilde{U}_{0}\|_{L^2(\widetilde{\Omega}^0)} \quad (\operatorname{supp}\{\widetilde{U}_{0}\}\subset\subset\widetilde{\Omega}_0) \notag \\
      &\leq e^{M_A T {\color{black}{/2}}}\|\widetilde{U}_{0}\|_{L^2(\widetilde{\Omega}^0)}
    \end{align}
    is valid. 
    {\color{red}{Here we use Lemma \ref{Lemma-psi-Omega-Rd-LP}. }}
    \\
    This leads to the fundamental estimate:
    \begin{align}\label{eq-U*-U0-L2}
      \|U^*\|_{L^2(\widetilde{\Omega}^{*})}\leq\|W^*\|_{L^2(\mathbb{R}^d)}=\|f^0\|_{H^{\prime}(\mathbb{R}^d)} \leq e^{M_A t^{*} {\color{black}{/2}}}\|\widetilde{U}_0\|_{L^2(\widetilde{\Omega}^{0})} \leq e^{M_A T {\color{black}{/2}}}\|\widetilde{U}_0\|_{L^2(\widetilde{\Omega}^{0})}<\infty, ~\forall t^*\in(0,T] 
    \end{align}
    establishing
    \begin{align}\label{eq-U-L1t-L2x}
      \|\mathcal{U}\|_{L^{1}((0,T],H(\widetilde{\Omega}^{\lambda}))} &= \int_{0}^{T}\|U^{*}\|_{L^2(\widetilde{\Omega}^{*})}\mathrm{~d}t^{*}
      \leq \int_{0}^{T}e^{M_A T {\color{black}{/2}}}\|\widetilde{U}_0\|_{L^2(\widetilde{\Omega}^{0})}\mathrm{~d}t^{*}
      = T e^{M_A T {\color{black}{/2}}} \|\widetilde{U}_0\|_{L^2(\widetilde{\Omega}^{0})} \notag \\
      &<\infty.   
    \end{align}

  \item \textbf{Step4.\ Verification of the Initial Condition. }
    \par
    In \textbf{Step 1}, by setting $t^{*}=0$, we obtain:
    \begin{align}
      \exists ! \mathcal{W}^{0} \in H(\mathbb{R}^d) \quad \text{such that} \quad 
      \int_{\widetilde{\Omega}^{0}} \mathcal{W}^{0}\psi^0  \mathrm{d}\mathbf{x} = \int_{\widetilde{\Omega}^{0}} \widetilde{U}_{0}\psi^{0}  \mathrm{d}\mathbf{x}, \quad 
      \forall \psi^0 \in H_0(\widetilde{\Omega}^{0}). 
    \end{align}
    Defining:
    \begin{align}\label{eq-Riesz-U0-W0}
      U^{0} = \mathcal{W}^{0}(\mathbf{x}) \cdot \chi_{\widetilde{\Omega}^{0}}(\mathbf{x}) \in H(\widetilde{\Omega}^{0}), 
    \end{align}
    we have:
    \begin{align}
      \int_{\widetilde{\Omega}^{0}} U^{0}\psi^0  \mathrm{d}\mathbf{x} = \int_{\widetilde{\Omega}^{0}} \widetilde{U}_{0}\psi^{0}  \mathrm{d}\mathbf{x}, \quad 
      \forall \psi^0 \in H_0(\widetilde{\Omega}^{0}),
    \end{align}
    which implies:
    \begin{align}
      \int_{\widetilde{\Omega}^{0}} (U^{0} - \widetilde{U}_{0})\psi^0  \mathrm{d}\mathbf{x} = 0, \quad 
      \forall \psi^0 \in H_0(\widetilde{\Omega}^{0}). 
    \end{align}
    By the \textbf{Fundamental Lemma of the Calculus of Variations}, we conclude:
    \[
    U^{0} \overset{\text{a.e.}}{=} \widetilde{U}_{0} \quad \text{in} \quad \widetilde{\Omega}^{0}.
    \]
    We can therefore redefine $U^0 = \widetilde{U}_{0}\in L^2(\widetilde{\Omega}^{0})$, ensuring that the abstract function $\mathcal{U}(\lambda)$ satisfies the initial condition in Definition \ref{def-WeakSolution}. 
\end{itemize}
In summary, 
\begin{align}
  \mathcal{U}(\lambda)\in L^{1}([0,T],H(\widetilde{\Omega}^{\lambda}))
\end{align}
and is identified as the solution to the integral invariant model in the sense of Definition \ref{def-WeakSolution}. 
\end{proof}
\begin{remark}[Riesz Representation Theorem]
  Let \( H \) be a Hilbert space and \( f \in H^{\prime} \) be a continuous linear functional on \( H \). Then, there exists a unique vector \( y \in H \) such that for all \( x \in H \),
  \(
  f(x) = \langle x, y \rangle,
  \)
  and the norm of \( f \) satisfies \( \|f\| = \|y\| \).
\end{remark}
\begin{remark}[Definition and Properties of Operator Norm]
  Let \(\mathcal{A}\) be a bounded linear operator from a normed space \(X\) to a normed space \(Y\). The positive number \(\|\mathcal{A}\| = \sup_{\substack{x \neq 0}} \frac{\|\mathcal{A}x\|}{\|x\|}\) is called the norm of \(\mathcal{A}\). 
  An equivalent formulation of this definition is
  \[
  \|\mathcal{A}\| = \sup \{\|\mathcal{A}x\| : x \in X, \|x\| \leq 1\} = \sup \{\|\mathcal{A}x\| : x \in X, \|x\| = 1\}.
  \]
\end{remark}
\begin{remark}[Stability of Integral Invariant Model]
  The stability property is manifested in Eq. \eqref{eq-U*-U0-L2}. 
  A generalized treatment of this stability, particularly the continuous dependence on the initial data $\widetilde{U}_0$ of the integral invariant model\eqref{eq-ComprehensiveIntegralInvariantModel}, 
  will be provided in Theorem \ref{Thm-WeakSolution-stability}, Section \ref{SubSec-Stability-Uniqueness}.   
\end{remark}
\par
In the course of proving Theorem \ref{Thm-WeakSolution-Existence}, while the uniqueness of the representing element $\mathcal{W}^*$ provided by the Riesz representation theorem ensures the uniqueness of $U^*$, 
this merely establishes that the mapping from $t^*$ to $U^*$ constitutes a single-valued mapping (i.e., a proper function in the strict sense), without precluding the potential existence of other abstract functions satisfying Definition \ref{def-WeakSolution}. 
Thus, the uniqueness of the integral invariant model's solution is to be demonstrated through its continuous dependence on initial data (stability). Accordingly, the stability analysis of the model\eqref{eq-ComprehensiveIntegralInvariantModel} in Definition \ref{def-WeakSolution} will first be conducted, after which the proof of uniqueness will be presented.

\subsection{Stability and Uniqueness of Integral Invariant Model\eqref{eq-ComprehensiveIntegralInvariantModel}}\label{SubSec-Stability-Uniqueness}
\begin{theorem}[Continuous Dependence on Initial Data (Stability) of The Integral Invariant Model\eqref{eq-ComprehensiveIntegralInvariantModel}]\label{Thm-WeakSolution-stability}
  Let \( \Omega \) be a bounded open domain in \( \mathbb{R}^d \). Define \( \widetilde{\Omega}(t) = \operatorname{D}^{\mathbf{A}}_{T\rightarrow t}(\Omega) \) for \( t \in [0, T] \). 
  Assume \( \widetilde{U}_0(\mathbf{x}), \widetilde{V}_0(\mathbf{x}) \in H_0(\widetilde{\Omega}(0)) \) with \( \widetilde{U}_0 \neq \widetilde{V}_0 \).  
  Then, the solutions \( U(\mathbf{x}, t) \) and \( V(\mathbf{x}, t) \) to the integral invariant model\eqref{eq-ComprehensiveIntegralInvariantModel} with initial data  \(\widetilde{U}_0(\mathbf{x})\) and \(\widetilde{V}_0(\mathbf{x})\) , respectively, satisfy
  \begin{align}
    \|U(\cdot\ ,t) - V(\cdot\ ,t)\|_{L^2(\widetilde{\Omega}(t))}
    \leq 
    e^{M_A T / 2} \|\widetilde{U}_{0} - \widetilde{V}_{0}\|_{L^{2}(\widetilde{\Omega}(0))}, \quad
    \forall t \in [0, T].
  \end{align}
\end{theorem} 

\begin{proof}
  Based on the formulation \eqref{eq-WeakSol-EquivalentForm2}, which is equivalent to Definition \ref{def-WeakSolution}, we consider
  \begin{align}
  \int_{\widetilde{\Omega}(t)} U(\mathbf{x}, t) \psi(\mathbf{x}, t) \, \mathrm{d}\mathbf{x} = \int_{\widetilde{\Omega}(0)} \widetilde{U}_{0}(\mathbf{x}) \psi(\mathbf{x}, 0) \, \mathrm{d}\mathbf{x}, \quad \forall t \in [0, T].
  \end{align}
  Using \eqref{eq-psi-Psi} again, we obtain
  \begin{align}\label{eq-WeakSolution-stability-U}
  \int_{\widetilde{\Omega}(t)} U(\mathbf{x}, t) \Psi\left(\mathbf{x} + \int_{t}^{T} \mathbf{A}(\mathbf{s}(\tau), \tau) \, \mathrm{d}\tau\right) \, \mathrm{d}\mathbf{x} = \int_{\widetilde{\Omega}(0)} \widetilde{U}_{0}(\mathbf{x}) \Psi\left(\mathbf{x} + \int_{0}^{T} \mathbf{A}(\mathbf{s}(\tau), \tau) \, \mathrm{d}\tau\right) \, \mathrm{d}\mathbf{x}.
  \end{align}
  According to Theorem \ref{Thm-WeakSolution-Existence}, there exists a \(V(\mathbf{x}, t) \in L^{1}([0,T],H(\widetilde{\Omega}(t)))\) satisfying the Definition \ref{def-WeakSolution} or equivalent formulation\eqref{eq-WeakSol-EquivalentForm2} with initial value \(\widetilde{V}_0(\mathbf{x})\), i.e., for all \(t \in [0, T]\),
  \begin{align}\label{eq-WeakSolution-stability-W}
  \int_{\widetilde{\Omega}(t)} V(\mathbf{x}, t) \Psi\left(\mathbf{x} + \int_{t}^{T} \mathbf{A}(\mathbf{s}(\tau), \tau) \, \mathrm{d}\tau\right) \, \mathrm{d}\mathbf{x} = \int_{\widetilde{\Omega}(0)} \widetilde{V}_{0}(\mathbf{x}) \Psi\left(\mathbf{x} + \int_{0}^{T} \mathbf{A}(\mathbf{s}(\tau), \tau) \, \mathrm{d}\tau\right) \, \mathrm{d}\mathbf{x}.
  \end{align}
  {\color{black}{$(\text{Note: } \widetilde{U}_{0} - \widetilde{V}_{0} \text{ is the initial perturbation.})$}}
\\
Subtracting Eq. \eqref{eq-WeakSolution-stability-W} from Eq. \eqref{eq-WeakSolution-stability-U}, we obtain
{\scriptsize
\begin{align}
\underbrace{\int_{\widetilde{\Omega}(t)} (U(\mathbf{x}, t) - V(\mathbf{x}, t)) \Psi\left(\mathbf{x} + \int_{t}^{T} \mathbf{A}(\mathbf{s}(\tau), \tau) \, \mathrm{d}\tau\right) \, \mathrm{d}\mathbf{x}}_{LHS} = \underbrace{\int_{\widetilde{\Omega}(0)} (\widetilde{U}_{0}(\mathbf{x}) - \widetilde{V}_{0}(\mathbf{x})) \Psi\left(\mathbf{x} + \int_{0}^{T} \mathbf{A}(\mathbf{s}(\tau), \tau) \, \mathrm{d}\tau\right) \, \mathrm{d}\mathbf{x}}_{RHS}.
\end{align}
}
\par \noindent
For convenience, we define
\begin{align}\label{eq-Delta-t1-t2}
\Delta_{(\mathbf{A}, t, T)} := \int_{t}^{T} \mathbf{A}(\mathbf{s}(\tau), \tau) \, \mathrm{d}\tau, \quad t \in [0, T]. 
\end{align}
Regarding the variable \( t \)  as a fixed moment in time, we choose
{\color{black}{
\begin{align}\label{eq-Stability-testFun-Psi}
\Psi(\mathbf{x})=(U(\mathbf{x}-\Delta_{(\mathbf{A},t,T)}\ ,t)
-V(\mathbf{x}-\Delta_{(\mathbf{A},t,T)}\ ,t))
\cdot\chi_{{\widetilde{\Omega}(t)}}(\mathbf{x}-\Delta_{(\mathbf{A},t,T)}), \ \mathbf{x}\in\mathbb{R}^d,
\end{align}
}} 
where
\begin{align}
\chi_{\widetilde{\Omega}(t)}(\omega) =
\begin{cases}
1, & \omega \in \widetilde{\Omega}(t), \\
0, & \omega \notin \widetilde{\Omega}(t),
\end{cases}
\end{align}
and it follows that
\begin{align}
\|\Psi(\cdot,t)\|_{L^2(\Omega)} = \|U(\cdot, t) - V(\cdot, t)\|_{L^2(\widetilde{\Omega}(t))},\ \forall t\in[0,T].
\end{align}
{\color{black}{{\color{black}{Note: $\forall \mathbf{x}\in\Omega$, $\mathbf{x}-\Delta_{(\mathbf{A},t,T)}\in\widetilde{\Omega}(t)$ 
while $\forall \mathbf{x}\in\mathbb{R}^{d}\setminus\Omega$, $\mathbf{x} - \Delta_{(\mathbf{A}, t, T)} \notin \widetilde{\Omega}(t)$. }}}}
\\ 
{\color{black}{It is consequently established that $\Psi\in H_{0}(\Omega)$. }}Thus, we have
\begin{align}
LHS &= \int_{\widetilde{\Omega}(t)} (U(\mathbf{x}, t) - V(\mathbf{x}, t)) \cdot (U(\mathbf{x} + \Delta_{(\mathbf{A}, t, T)} - \Delta_{(\mathbf{A}, t, T)}, t) \notag \\
&\quad - V(\mathbf{x} + \Delta_{(\mathbf{A}, t, T)} - \Delta_{(\mathbf{A}, t, T)}, t)) \cdot \chi_{\widetilde{\Omega}(t)}(\mathbf{x} + \Delta_{(\mathbf{A}, t, T)} - \Delta_{(\mathbf{A}, t, T)}) \, \mathrm{d}\mathbf{x} \notag \\
&= \int_{\widetilde{\Omega}(t)} |U(\mathbf{x}, t) - V(\mathbf{x}, t)|^2 \cdot \chi_{\widetilde{\Omega}(t)}(\mathbf{x}) \, \mathrm{d}\mathbf{x} \notag \\
&= \int_{\widetilde{\Omega}(t)} |U(\mathbf{x}, t) - V(\mathbf{x}, t)|^2 \cdot 1 \, \mathrm{d}\mathbf{x} \notag \\
&= \|U(\cdot, t) - V(\cdot, t)\|^2_{L^2(\widetilde{\Omega}(t))}, \label{eq-LHS} \\
\hfill \notag \\
RHS &\leq \int_{\widetilde{\Omega}(0)} \left|(\widetilde{U}_{0} - \widetilde{V}_{0}) \cdot \operatorname{\Phi}^{\mathbf{A}}_{T \to 0}[\Psi](\mathbf{x})\right| \, \mathrm{d}\mathbf{x} \notag \\
&\leq \|\widetilde{U}_{0} - \widetilde{V}_{0}\|_{L^2(\widetilde{\Omega}(0))} \cdot \left\|\operatorname{\Phi}^{\mathbf{A}}_{T \to 0}[\Psi](\mathbf{x})\right\|_{L^2(\widetilde{\Omega}(0))}. \label{eq-RHS}
\end{align}
Take $P=2$ in the Lemma \ref{Lemma-psi-Omega-Rd-LP}, we have 
\[
  \left\|\operatorname{\Phi}^{\mathbf{A}}_{T \to 0}[\Psi](\mathbf{x})\right\|_{L^2(\widetilde{\Omega}(0))} \leq  e^{M_A T / 2} \|\Psi\|_{L^2(\Omega)} = e^{M_A T / 2} \|U(\cdot, t) - V(\cdot, t)\|_{L^2(\widetilde{\Omega}(t))}.
\]
Since \(LHS = RHS\), it follows that
\begin{align}\label{eq-U-fluctuation-uniform-bounded}
  \|U(\cdot, t) - V(\cdot, t)\|_{L^2(\widetilde{\Omega}(t))} \leq e^{M_A T / 2} \|\widetilde{U}_{0} - \widetilde{V}_{0}\|_{L^2(\widetilde{\Omega}(0))}, ~\forall t \in [0, T].
\end{align}
\end{proof}
\begin{remark}[$L^2$-Norm Uniform Boundedness of The Fluctuation in Solutions]
  From Equation \eqref{eq-U-fluctuation-uniform-bounded}, since the control factor $e^{M_A T / 2}$ is a constant independent of $t$, 
  the fluctuation in solutions $U(\cdot, t) - V(\cdot, t)$ of the integral invariant model \eqref{eq-ComprehensiveIntegralInvariantModel}—induced by initial value perturbations $\widetilde{U}_{0} - \widetilde{V}_{0}$—remains \textbf{uniformly bounded} in the sense of $L^2$-norm at all subsequent times. 
\end{remark}

\begin{theorem}[Uniqueness of The Solution Satisfying Definition \ref{def-WeakSolution}]
The solution $U(\mathbf{x},t)$ 
of the integral invariant model\eqref{eq-ComprehensiveIntegralInvariantModel} is unique.
\end{theorem}
\begin{proof}{\color{red}{{
  The uniqueness follows directly from the stability estimate in Theorem \ref{Thm-WeakSolution-stability}. By setting $\widetilde{U}_{0} = \widetilde{V}_{0}$, we obtain $\|U(\cdot, t) - V(\cdot, t)\|_{L^2(\widetilde{\Omega}(t))} \leq 0$ for all $t \in [0, T]$, which implies $U(x, t) = V(x, t)$ almost everywhere in $\widetilde{Q}_T$. 
  }}}
\end{proof}

\section{Higher $L^q(1 \leq q \leq \infty)$-Temporal Integrability of Solutions to the Integral Invariant Model on $[0,T]$}\label{sec-WeakSol-regularity-t-Linfty}
\begin{theorem}\label{Thm-U0-L2x-U-LqtL2x}
  When $\widetilde{U}_0 \in L^2(\widetilde{\Omega}^{0})$, the solution of the integral invariant model \eqref{eq-ComprehensiveIntegralInvariantModel} can achieve higher integrability on $[0,T]$, specifically satisfying
  \begin{align}
    U(\mathbf{x},t) \in L^q([0,T],H(\widetilde{\Omega}(t))), ~1 \leq q \leq \infty, 
  \end{align} 
  where 
  \begin{align}
    \|U(\mathbf{x},t)\|_{L^{\infty}([0,T],H(\widetilde{\Omega}(t)))}=\lim_{q \to \infty}\|U(\mathbf{x},t)\|_{L^{q}([0,T],H(\widetilde{\Omega}(t)))}.   
  \end{align}
\end{theorem}
\begin{proof}
  The case $q=1$ is given by equation \eqref{eq-U-L1t-L2x} in the existence proof. 
  We now prove the remaining cases in two steps: first for $2 \leq q < \infty$, then separately for $q=\infty$.  
  \begin{itemize}
    \item \textbf{Step 1. $2 \leq q < \infty$. } \par
    Since $\widetilde{U}_0 \in L^2(\widetilde{\Omega}^{0})$, the estimate \eqref{eq-U*-U0-L2} from the existence proof remains valid. Thus
    \begin{align*}
      \|U^*\|^q_{L^2(\widetilde{\Omega}^{*})} \leq e^{q M_A T {\color{black}{/2}}}\|\widetilde{U}_0\|^q_{L^2(\widetilde{\Omega}^{0})} < \infty, ~\forall t^* \in (0,T],   
    \end{align*}
    which implies
    \begin{align}\label{eq-Lqt-L2}
      \|\mathcal{U}(\lambda)\|_{L^{q}((0,T],H(\widetilde{\Omega}^{\lambda}))} 
      &= \left(\int_{0}^{T} \|U^{*}(\mathbf{x})\|^q_{L^2(\widetilde{\Omega}^{*})} \mathrm{d}t^{*}\right)^{1/q} \notag \\
      &\leq \left(\int_{0}^{T} e^{q M_A T {\color{black}{/2}}} \|\widetilde{U}_0(\mathbf{x})\|^q_{L^2(\widetilde{\Omega}^{0})} \mathrm{d}t^{*}\right)^{1/q} 
      = T^{1/q} e^{M_A T {\color{black}{/2}}} \|\widetilde{U}_0(\mathbf{x})\|_{L^2(\widetilde{\Omega}^{0})} < \infty.  
    \end{align}

    \item \textbf{Step 2. $q=\infty$. } \par
    Let ${\color{black}{\min(T,1)}} \geq \delta > 0$. From equation \eqref{eq-Lqt-L2} we observe that
    \begin{align}
      \|\mathcal{U}(\lambda)\|_{L^{q}((0,T],H(\widetilde{\Omega}^{\lambda}))} \leq \delta^{-1} T e^{M_A T {\color{black}{/2}}} \|\widetilde{U}_0(\mathbf{x})\|_{L^2(\widetilde{\Omega}^{0})}, \quad \forall 1 \leq q < \infty,
    \end{align}
    meaning there exists a finite positive constant $C$ (e.g., $C = \delta^{-1}$) independent of $q$ ({\color{black}{but dependent on $T$}}) such that the $L^{q}((0,T],H(\widetilde{\Omega}^{\lambda}))$-norm of the abstract function $\mathcal{U}(\lambda)$ is uniformly bounded for all $q$. 
    By Lemma \ref{lemma-Lp-Linfty}, we immediately obtain
    \begin{align*}
      U(\mathbf{x},t) \in L^{\infty}([0,T],H(\widetilde{\Omega}(t)))
    \end{align*}
    with
    \begin{align*}
      \|U(\mathbf{x},t)\|_{L^{q}([0,T],H(\widetilde{\Omega}(t)))} \xrightarrow[q \to \infty]{} \|U(\mathbf{x},t)\|_{L^{\infty}([0,T],H(\widetilde{\Omega}(t)))}. 
    \end{align*}
  \end{itemize}
\end{proof}

\section{Conclusion}\label{sec-Conclusion}
In this paper, a comprehensive mathematical formulation for the integral invariant model, 
derived from the linear transport equation and its adjoint equation, is provided in the form of a Cauchy initial value problem. 
The existence of this model's solution is established through the application of the Riesz representation theorem and abstract function theory in $L^{1}([0,T],L^2(\widetilde{\Omega}(t)))$. 
Subsequently, its stability and uniqueness are demonstrated by the strategic selection of the test function \(\Psi\). 
Following the well-posedness analysis of this integral invariant model, we further investigate its regularity properties.  
A more detailed analysis reveals that when initial value $\widetilde{U}_0(\mathbf{x}) \in L^2(\widetilde{\Omega}(0))$, 
the integrability in time of the integral invariant model can indeed be made arbitrarily high, specifically  
\( U(\mathbf{x},t) \in L^q([0,T],L^2(\widetilde{\Omega}(t))), ~1 \leq q \leq \infty \), 
where for the case $q=\infty$, we utilize the established Lemma \ref{lemma-Lp-Linfty} from previous work\cite{ref-CHenSX18-lemma}. 
\par
Future work will focus on investigating the differentiability of the integral invariant model, with the goal of achieving the regularity $U(\mathbf{x},t) \in H^1([0,T],H^1(\widetilde{\Omega}(t)))$ to the greatest extent possible. 

\section*{Acknowledgments}
The authors would like to thank the anonymous referees for their very valuable comments and suggestions. 

\section*{Declarations} 
\textbf{Funding}
There was no funding for this paper. \par
\textbf{Conflict of interest}
The authors have no conflict of interest to declare that are relevant to the content of this article.



\appendix
\section{Differentiation formula for variable-limit integrals(Leibniz Rule)} \label{appendix-Differentiation of a Definite Integral with Variable Limits}
\(\Omega({\color{black}{t}})\) denotes a time-dependent bounded domain, 
and at time \( t^* \), \(\Omega({\color{black}{t}})\) is fixed as \(\Omega^{\star}_t\), i.e., \(\Omega^{\star}_t = \Omega({\color{black}{t^{\star}}})\). Let \(\frac{\mathrm{d}}{\mathrm{d} t}(\cdot)\) 
represent the material derivative (substantial derivative), and denote \(\mathbf{u} = \frac{\mathrm{d}\mathbf{x}}{\mathrm{d}t}\). 
The differentiation formula for integrals with variable limits (also known as the Leibniz rule) is given by: 
\begin{itemize}
  \item For a scalar function \( f(\mathbf{x}, t) : \mathbb{R}^d \times [0, +\infty] \longrightarrow \mathbb{R} \), 
  \begin{align}\label{eq-scalar-Differentiation of a Definite Integral with Variable Limits}
  \frac{\mathrm{d}}{\mathrm{d} t}\left[\int_{\Omega{\color{black}{(t)}}} f(\mathbf{x}, t) \, \mathrm{d} \mathbf{x} \right]
  = \int_{\Omega^{\star}_t} \frac{\partial f}{\partial t} + \nabla \cdot \left(f \cdot \frac{\mathrm{d} \mathbf{x}}{\mathrm{d}t}\right) \, \mathrm{d} \mathbf{x}
  = \int_{\Omega^{\star}_t} \frac{\partial f}{\partial t} + \nabla \cdot (f \cdot \mathbf{u}) \, \mathrm{d} \mathbf{x}.
  \end{align}
  \item For a vector function \(\mathbf{F}(\mathbf{x}, t): \mathbb{R}^d \times [0, +\infty] \longrightarrow \mathbb{R}^d\), 
  \begin{align}\label{eq-vector-Differentiation of a Definite Integral with Variable Limits}
    \frac{\mathrm{d}}{\mathrm{d} t}\left[\int_{\Omega{\color{black}{(t)}}} \mathbf{F}(\mathbf{x}, t) \mathrm{~d} \mathbf{x} \right]
    = \int_{\Omega^{\star}_t} \frac{\partial \mathbf{F}}{\partial t} + \nabla \cdot (\mathbf{F} \otimes \frac{\mathrm{d} \mathbf{x}}{\mathrm{d}t}) \mathrm{~d} \mathbf{x}
    = \int_{\Omega^{\star}_t} \frac{\partial \mathbf{F}}{\partial t} + \nabla \cdot (\mathbf{F} \otimes \mathbf{u}) \mathrm{~d} \mathbf{x},
  \end{align}
  where ``\(\otimes\)'' denotes the ``Kronecker product'' between vectors, defined as:
  \begin{align}
    \mathbf{a} \otimes \mathbf{b} = \mathbf{a} \mathbf{b}^{\mathrm{T}}
    = \left[\begin{array}{l}
    a_1 \\
    a_2 \\
    a_3
    \end{array}\right]
    \left[b_1,  b_2,  b_3\right]
    = \left[\begin{array}{lll}
    a_1 b_1 & a_1 b_2 & a_1 b_3 \\
    a_2 b_1 & a_2 b_2 & a_2 b_3 \\
    a_3 b_1 & a_3 b_2 & a_3 b_3
    \end{array}\right].
  \end{align}
\end{itemize}
\begin{remark}
  Starting from the equation \eqref{eq-scalar-Differentiation of a Definite Integral with Variable Limits}, and using the divergence theorem, we obtain
  \[
  \begin{aligned}
  \frac{\mathrm{d}}{\mathrm{d} t}\left[\int_{\Omega{\color{black}{(t)}}} f(\mathbf{X}, t) \, \mathrm{d} \mathbf{X} \right]
  &= \int_{\Omega^{\star}_t} \frac{\partial f}{\partial t} + \int_{\Omega^{\star}_t} \nabla \cdot (\mathbf{u} \cdot f) \, \mathrm{d} \mathbf{X} \\
  &= \int_{\Omega^{\star}_t} \frac{\partial f}{\partial t} + \int_{\partial\Omega^{\star}_t} \mathbf{n} \cdot (\mathbf{u} \cdot f) \, \mathrm{d} \sigma(\mathbf{X}) \\
  &= \int_{\Omega^{\star}_t} \frac{\partial f}{\partial t} + \int_{\partial\Omega^{\star}_t} f (\mathbf{n} \cdot \mathbf{u}) \, \mathrm{d} \sigma(\mathbf{X}),
  \end{aligned}
  \]
  where \(\mathbf{n}\) is the unit outward normal vector to \(\partial\Omega_t^{\star}\), 
  and the last equality is the \textbf{Reynolds transport theorem}. 
\end{remark}

\section{Some Application Examples} \label{appendix-A-HYP-Examples} 
We list several common numerical test cases used in algorithms such as SLDG, SLFV, and ELDG(Eulerian-Lagrangian Discontinuous Galerkin), analyzing the properties of their velocity fields $\mathbf{A}$ to demonstrate the practicality of our \textbf{first-order continuous differentiability} and \textbf{uniformly bounded gradient} conditions: 
\begin{itemize}
  \item 1D cases: 
  \begin{itemize}
    \item $u_t+u_x=0$: \\
    $A(x,t)=1$ satisfies $A \in C^1(\mathbb{R}\times[0,+\infty))$ and $\sup_{x \in \mathbb{R}, t \in [0, +\infty]}|\partial_x A|=0$.  
    \\
    \item $u_t+(\sin(t)u)_x=0$: \\
    $A(x,t)=\sin(t)$ satisfies $A \in C^1(\mathbb{R}\times[0,+\infty))$ and $\sup_{x \in \mathbb{R}, t \in [0, +\infty]}|\partial_x A|=0$.  
    \\
    \item $u_t+(\sin(x)u)_x=0$ :\\
    $A(x,t)=\sin(x)$ satisfies $A \in C^1(\mathbb{R}\times[0,+\infty))$ and  \\
    $\sup_{x \in \mathbb{R}, t \in [0, +\infty]}|\partial_x A|=\sup_{x \in \mathbb{R}, t \in [0, +\infty]}|\cos(x)| = 1$.  
  \end{itemize}
  \hfill \\
  \item 2D cases: 
  \begin{itemize}
    \item $U_t+U_x+U_y=0$: \\
    $\mathbf{A}(\mathbf{x},t)=(1,1)^{\mathrm{T}}$ satisfies $\mathbf{A} \in C^1(\mathbb{R}^2\times[0,+\infty))$, \\
    $\sup_{\mathbf{x} \in \mathbb{R}^2, t \in [0, \infty)} \left\| \nabla_{\mathbf{x}} \mathbf{A}(\mathbf{x}, t) \right\|_{\mathrm{op}} = 0$  
    and $\sup_{\mathbf{x} \in \mathbb{R}^d, t \in [0, +\infty)} \left| \nabla_{\mathbf{x}} \cdot \mathbf{A}(\mathbf{x}, t) \right| = 0$. 
    \\
    \item Rigid body rotation: $U_t-(yU)_x+(xU)_y=0$: \\
    $\mathbf{A}(\mathbf{x},t)=(-y,x)^{\mathrm{T}}$ satisfies $\mathbf{A} \in C^1(\mathbb{R}^2\times[0,+\infty))$, \\
    $\sup_{\mathbf{x} \in \mathbb{R}^2, t \in [0, \infty)} \left\| \nabla_{\mathbf{x}} \mathbf{A}(\mathbf{x}, t) \right\|_{\mathrm{op}} 
    = \sup_{\mathbf{x} \in \mathbb{R}^2, t \in [0, \infty)}
    \left\|\left(\begin{array}{ll}
      0,-1 \\
      1,0
    \end{array}\right)\right\|_{\mathrm{op}}
    =1$  \\
    and $\sup_{\mathbf{x} \in \mathbb{R}^d, t \in [0, +\infty)} \left| \nabla_{\mathbf{x}} \cdot \mathbf{A}(\mathbf{x}, t) \right| = 0$. 
    \\
    \item Swirling deformation: 
    $$U_t-\left(2 \pi \cos ^2\left(\frac{x}{2}\right) \sin (y) g(t) U\right)_x+\left(2 \pi \sin (x) \cos ^2\left(\frac{y}{2}\right) g(t) U\right)_y=0, $$  
    where $g(t)=\cos (\pi t / T)$:  \\
    $\mathbf{A}(\mathbf{x},t)=
    \begin{pmatrix}
      2 \pi \cos ^2\left(\frac{x}{2}\right) \sin (y) \cos (\pi t / T) \\
      2 \pi \sin (x) \cos ^2\left(\frac{y}{2}\right) \cos (\pi t / T)
    \end{pmatrix}$ 
    satisfies 
    $\mathbf{A} \in C^1(\mathbb{R}^2\times[0,+\infty))$, \\
    \begin{align*}
      &\sup_{\mathbf{x} \in \mathbb{R}^2, t \in [0, \infty)} \left\| \nabla_{\mathbf{x}} \mathbf{A}(\mathbf{x}, t) \right\|_{\mathrm{op}} = \\
      &\sup_{\mathbf{x} \in \mathbb{R}^2, t \in [0, \infty)} 
      \left\| 
        \begin{pmatrix}
          -\pi \sin(x)\sin(y)\cos(\pi t / T) &,&  2\pi\cos^2\left(\frac{x}{2}\right)\cos(y)\cos(\pi t / T)\\
          2\pi\cos^2\left(\frac{y}{2}\right)\cos(x)\cos(\pi t / T) &,& -\pi \sin(x)\sin(y)\cos(\pi t / T)
        \end{pmatrix}
      \right\|_{\mathrm{op}} = 2\pi
    \end{align*}
    and \\
    $$\sup_{\mathbf{x} \in \mathbb{R}^d, t \in [0, +\infty)} \left| \nabla_{\mathbf{x}} \cdot \mathbf{A}(\mathbf{x}, t) \right| = \sup_{\mathbf{x} \in \mathbb{R}^d, t \in [0, +\infty)} \left| 2\pi\sin(x)\sin(y)\cos(\pi t/T) \right| = 2\pi.$$  
  \end{itemize}
\end{itemize}
\end{document}